\documentclass{amsart}

\usepackage{amssymb,amsfonts, amsmath, amsthm}
\usepackage[all,arc]{xy}
\usepackage{enumerate}
\usepackage{mathrsfs}
\usepackage{mathtools}
\usepackage[mathscr]{euscript}
\usepackage{array}
\usepackage{appendix}

\usepackage{tikz}
\usepackage{mathpazo}
\usepackage{tikz-cd}
\usepackage{textcomp}
\usepackage[pagebackref, colorlinks=true, linkcolor=blue, citecolor = red]{hyperref}
 \renewcommand*{\backref}[1]{}
 \renewcommand*{\backrefalt}[4]{({%
     \ifcase #1 Not cited.%
           \or On p.~#2%
           \else On pp.~#2%
     \fi%
     })}

\usepackage{geometry}
\usepackage{todonotes}

\usepackage[nameinlink,capitalise,noabbrev]{cleveref}
\crefname{subsection}{Subsection}{Subsection}

\usetikzlibrary{patterns}

\makeatletter
\newcommand{\setword}[2]{%
  \phantomsection
  #1\def\@currentlabel{\unexpanded{#1}}\label{#2}%
}
\makeatother

\newcommand{\s}{\mathscr{S}}
\renewcommand{\ss}{s\mathscr{S}}
\newcommand{\sss}{ss\mathscr{S}}

\newcommand{\sset}{s\set}
\newcommand{\ssset}{ss\set}

\newcommand{\C}{\mathscr{C}}

\newcommand{\D}{\mathscr{D}}
\newcommand{\E}{\mathscr{E}}
\newcommand{\J}{\mathscr{J}}
\renewcommand{\L}{\mathcal{L}}
\newcommand{\M}{\mathcal{M}}
\newcommand{\N}{\mathcal{N}}

\newcommand{\reb}{]}
\newcommand{\leb}{[}

\newcommand{\set}{\mathscr{S}\text{et}}
\newcommand{\cat}{\mathscr{C}\text{at}}

\newcommand{\Forget}{\mathrm{Unm}}
\newcommand{\id}{\mathrm{id}}
\newcommand{\ds}{\displaystyle}

\newcommand{\inc}{\mathrm{inc}}
\newcommand{\iplus}{i_1^+}
\newcommand{\pplus}{p_1^+}
\newcommand{\tplus}{t^+}

\newcommand{\Hom}{\mathrm{Hom}}
\newcommand{\Map}{\mathrm{Map}}

\newcommand{\Fun}{\mathrm{Fun}}

\newcommand{\comma}{,}

\newcommand{\rotatebot}{\rotatebox{270}{$\bot$}}
\newcommand{\rotatetiltbot}{\rotatebox[origin=r]{40}{$\bot$}}
\newcommand{\rotateothertiltbot}{\rotatebox[origin=r, units=-360]{40}{$\bot$}}

\newcommand{\Seg}{\mathcal{S}\mathrm{eg}}
\newcommand{\Comp}{\mathcal{C}\mathrm{omp}}
\newcommand{\Sp}{\mathrm{Sp}}
\newcommand{\Cop}{\mathcal{C}\mathrm{op}}
\newcommand{\sCop}{\mathrm{s}\Cop}

\newcommand{\GrothFib}{\mathscr{G}\mathrm{roth}\mathscr{F}\mathrm{ib}}

\newcommand{\ordered}[1]{< #1 >}

\newcommand{\St}{\mathrm{St}}
\newcommand{\Un}{\mathrm{Un}}

\newcommand{\sSt}{\mathrm{sSt}}
\newcommand{\sUn}{\mathrm{sUn}}

\newcommand{\adjun}[4]{
\begin{tikzcd}[row sep=0.5in, column sep=0.5in]
 #1  \arrow[r, shift left=1.8, "#3", "\bot"'] \pgfmatrixnextcell
 #2 \arrow[l, shift left=1.8, "#4"] 
\end{tikzcd}
}

\newcommand{\pbsq}[8]{
  \begin{tikzcd}[row sep=0.5in, column sep=0.5in]
    #1 \arrow[r, "#5"] \arrow[d, "#6"'] \arrow[dr, phantom, "\ulcorner", very near start]
    \pgfmatrixnextcell #2 \arrow[d, "#7"] \\
    #3 \arrow[r, "#8"']
    \pgfmatrixnextcell #4
  \end{tikzcd}
}

\newcommand{\liftsq}[8]{
  \begin{tikzcd}[row sep=0.5in, column sep=0.5in]
    #1 \arrow[r, "#5"] \arrow[d, "#6"']
    \pgfmatrixnextcell #2 \arrow[d, "#7"] \\
    #3 \arrow[r, "#8"] \arrow[ur, dashed]
    \pgfmatrixnextcell #4
  \end{tikzcd}
}

\newcommand{\simpset}[7]{

 \begin{tikzcd}[row sep=0.5in, column sep=0.5in]
   #1 \arrow[r, shorten >=1ex,shorten <=1ex]
   \pgfmatrixnextcell #2 
   \arrow[l, shift left=1.2, "#5"] \arrow[l, shift right=1.2, "#4"'] 
   \arrow[r, shift right, shorten >=1ex,shorten <=1ex ] \arrow[r, shift left, shorten >=1ex,shorten <=1ex] 
   \pgfmatrixnextcell #3 
   \arrow[l] \arrow[l, shift left=2, "#7"] \arrow[l, shift right=2, "#6 "'] 
   \arrow[r, shorten >=1ex,shorten <=1ex] \arrow[r, shift left=2, shorten >=1ex,shorten <=1ex] \arrow[r, shift right=2, shorten >=1ex,shorten <=1ex]
   \pgfmatrixnextcell \cdots 
   \arrow[l, shift right=1] \arrow[l, shift left=1] \arrow[l, shift right=3] \arrow[l, shift left=3] 
 \end{tikzcd}
}

\newcommand{\arrowref}[2]{\hyperref[#1]{#2}}
\newcommand{\somethingelse}[2]{ \begin{tikzcd}[row sep=0in] #1 \\ #2 \end{tikzcd}}

\newtheorem{theone}[equation]{Theorem}
\newtheorem{lemone}[equation]{Lemma}
\newtheorem{propone}[equation]{Proposition}
\newtheorem{corone}[equation]{Corollary}

\theoremstyle{definition}
\newtheorem{defone}[equation]{Definition}
\newtheorem{exone}[equation]{Example}

\theoremstyle{remark}
\newtheorem{remone}[equation]{Remark}
\newtheorem{notone}[equation]{Notation}

\newtheoremstyle{TheoremNum}
{}{}              
{\itshape}                      
{}                              
{\bfseries}                     
{.}                             
{ }                             
{\thmname{#1}\thmnote{ \bfseries #3}}
\theoremstyle{TheoremNum}

\numberwithin{equation}{section}

\makeatletter
\def\@seccntformat#1{%
  \expandafter\ifx\csname c@#1\endcsname\c@section\else
  \csname the#1\endcsname\quad
  \fi}
\makeatother

\title{Quasi-Categories vs. Segal Spaces: Cartesian Edition}

\author{Nima Rasekh}
\address{{\'E}cole Polytechnique F{\'e}d{\'e}rale de Lausanne, SV BMI UPHESS, Station 8, CH-1015 Lausanne, Switzerland}
\email{nima.rasekh@epfl.ch}
\date{August 2021}

\begin{document}

\begin{abstract}
We prove that four different ways of defining Cartesian fibrations and the Cartesian model structure
are all Quillen equivalent:
\begin{enumerate}
 \item On marked simplicial sets (due to Lurie \cite{lurie2009htt}),
 \item On bisimplicial spaces (due to deBrito \cite{debrito2018leftfibration}),
 \item On bisimplicial sets,
 \item On marked simplicial spaces. 
\end{enumerate}
 The main way to prove these equivalences is by using the Quillen equivalences between quasi-categories and 
 complete Segal spaces as defined by Joyal-Tierney and the straightening construction due to Lurie.
\end{abstract}

\maketitle
\addtocontents{toc}{\protect\setcounter{tocdepth}{1}}

\tableofcontents

\section{Introduction}\label{Sec Introduction}

\subsection{Functors via Fibrations}
For a given category $\C$, we are often interested in studying {\it pseudo-functors} of the form 
$$P: \C^{op} \to \cat$$
where $\cat$ is the (large) category of (small) categories. 
Such functors commonly arise in 
algebraic geometry (such as moduli problems \cite{sga4}), 
topos theory (such as internal toposes \cite{johnstone2002elephanti}), 
algebraic topology (such as moduli stacks in chromatic homotopy \cite{conrad2007moduliellipticcurves}), ... .

In certain mathematical situations, so-called {\it derived} or {\it homotopical mathematics}, 
the definition of a category is too strict and we need to work with a weaker notion of a category 
in which the usual axioms of a category only hold up to suitably compatible equivalences or ``homotopies" 
\cite{groth2010inftycategories}. The study of such 
categories is now usually called {\it the theory of $(\infty,1)$-categories} \cite{bergner2010survey}. 

Continuing our analogy with classical categories, the study of homotopical mathematics, 
such as {\it derived algebraic geometry} \cite{lurie2004dag}, {\it higher topos theory} \cite{lurie2009htt}, ... , 
again requires us to study $(\infty,1)$-functors 
$$P: \C^{op} \to \cat_{(\infty,1)},$$
however, unlike the classical case, this is quite challenging and often not feasible. Indeed, because of the higher coherence conditions, 
any such functor $P$ requires an infinite tower of data, which usually prevents us from giving explicit descriptions. 

Fortunately, there is an alternative approach towards such functors in the context of categories: {\it fibrations}. A functor $p: \D \to \C$ is 
called a {\it Grothendieck fibration} if every morphism in $\C$, with specified lift of the target in $\D$, can be lifted to a 
$p$-Cartesian morphisms in $\D$. 

Here, a morphism $f: x \to y$ in $\D$ is $p$-Cartesian if the induced commutative square
\begin{center}
 \begin{tikzcd}
  \D_{/y} \arrow[r, "f^*"] \arrow[d, "p"] & \D_{/x} \arrow[d, "p"] \\
  \C_{/p(y)} \arrow[r, "p(f)"] & \C_{/p(x)} 
 \end{tikzcd}
\end{center}
is a pseudo-pullback of categories. 

Using this definition of fibration, Grothendieck \cite{grothendieck2003etalegroup} 
proved following equivalence between pseudo-functors and Grothendieck fibrations,
 $$\ds \int_\C: \Fun(\C^{op}, \cat) \xrightarrow{ \ \simeq \ } \GrothFib_\C,$$
now commonly called the {\it Grothendieck construction}, 
thus giving us a way to translate between Grothendieck fibrations and pseudo-functors valued in categories.

The fibrational approach is far more amenable to $(\infty,1)$-categorical techniques and thus can be generalized in a straightforward manner. 
This was done by Lurie \cite{lurie2009htt}. Using {\it quasi-categories}, which are simplicial sets that model $(\infty,1)$-categories 
\cite{boardmanvogt1973qcats,joyal2008notes,joyal2008theory}, he defined 
{\it Cartesian fibrations} as inner fibrations of simplicial sets that have sufficient $p$-Cartesian lifts. 

Following a long tradition in homotopy theory, he then proceeded to define a {\it model structure} \cite{quillen1967modelcats}
such that the fibrant objects were precisely the Cartesian fibrations over a fixed simplicial set $S$. However, he did not define the model structure on the category of simplicial sets over $S$, but rather on the category of {\it marked simplicial sets over $S$}, which are simplicial sets with a chosen subset of the set of $1$-simplices. Hence, despite the fact that the definition of a Cartesian fibration required no markings, the markings were a crucial aspect of its model structure. 
\footnote{It is in fact widely believed that it is not possible to define a model structure on simplicial sets over a given base, 
such that the fibrant objects are Cartesian fibrations and cofibrations are monomorphisms.}

Using this model structure, aptly named the {\it Cartesian model structure}, he then proves a Quillen equivalence \cite[Theorem 3.2.0.1]{lurie2009htt}
\begin{center}
 \adjun{(\sset^+_{/S})^{Cart}}{\Fun(\mathfrak{C}[S]^{op},(\sset^+)^{Cart})^{proj}}{\St^+_S}{\Un^+_S}
\end{center}
which is simply a technical correspondence between Cartesian fibrations and functors valued in $(\infty,1)$-categories, 
thus giving us a concrete way to transition between those, generalizing the Grothendieck construction from categories to 
$(\infty,1)$-categories. 

In the subsequent years, several authors have studied Cartesian fibrations. 
For example, Riehl and Verity study Cartesian fibrations in the context of an $\infty$-{\it cosmos}, 
which is a formal approach to $(\infty,1)$-categories \cite{riehlverity2017inftycosmos} motivated by formal category theory. 
Another example is the work by Ayala and Francis, who study the quasi-category of Cartesian fibrations from the 
perspective of {\it exponentiable fibrations} \cite{ayalafrancis2020fibrations}.
It should be noted, however, that neither approach results in a new model category for Cartesian fibrations. 
Rather they use the fact that the definition of $\infty$-cosmoi or quasi-categories are less strict to 
study Cartesian fibrations directly, without having to take an external approach.
 
\subsection{Cartesian Fibrations via Complete Segal Objects}
In this paper we propose an alternative approach towards Cartesian fibrations via {\it complete Segal objects} 
(rather than marked simplicial objects) which results in a new a model category of Cartesian fibrations. 
In order to understand this different approach it is necessary to understand complete Segal spaces, right fibrations and how
they can interact.

A {\it complete Segal space} \cite{rezk2001css} is a simplicial space 
\begin{center}
 \simpset{W_0}{W_1}{W_2}{}{}{}{}
\end{center}
that satisfies three sets of conditions (the Reedy, Segal and completeness conditions). 
These conditions imply that complete Segal spaces are another model of $(\infty,1)$-categories 
and in particular equivalent to quasi-categories \cite{joyaltierney2007qcatvssegal} and can in fact be seen as the standard model \cite{toen2005unicity}. 

The benefit of complete Segal spaces is that the conditions can all be expressed via certain finite limit diagrams \cite{rezk2010thetanspaces}. 
Hence, we can easily generalize a complete Segal space 
to a {\it complete Segal object} in any category (or model category or $(\infty,1)$-category) with finite limits.
Using our intuition from complete Segal spaces, we expect a complete Segal object in a category to play the role of an 
{\it internal $(\infty,1)$-category}.\footnote{This idea is in fact not new. As an example, see the notion of a {\it Rezk object} in 
\cite{riehlverity2017inftycosmos}.} 
So, in particular, a complete Segal object in the category of space-valued presheaves is 
precisely a presheaf valued in complete Segal spaces i.e. $(\infty,1)$-categories.

The discussion in the previous paragraph motivates us to study fibrations that correspond to presheaves valued in spaces. 
As spaces are special kinds of $(\infty,1)$-categories, we could simply think of them as special kinds of Cartesian fibrations. 
However, it turns out there is a direct approach as well: right fibrations \cite{joyal2008notes,joyal2008theory}. 
A {\it right fibration} is a map of simplicial sets that satisfies the right lifting condition with respect to squares 
\begin{center}
 \liftsq{\Lambda[n]_i}{R}{\Delta[n]}{S}{}{}{}{}
\end{center}
where $n \geq 0$, $0 < i \leq n$.
Unlike Cartesian fibrations we can in fact endow the category of simplicial sets over a fixed simplicial set with a model structure, 
the {\it contravariant model structure}, such that the fibrant objects are precisely the right fibrations. 
It was first proven by Lurie \cite{lurie2009htt}
(and later many other authors \cite{stevenson2017covariant,heutsmoerdijk2015leftfibrationi,heutsmoerdijk2016leftfibrationii}), 
that this model structure is Quillen equivalent to presheaves valued in spaces. 
Moreover, we can give an analogous definition of right fibrations and contravariant model structure in the context of simplicial spaces 
in a way that is Quillen equivalent to the contravariant model structure for simplicial sets \cite{rasekh2017left,debrito2018leftfibration}.

Combining the two previous paragraphs, we should expect that a complete Segal object in right fibrations, which has the form
\begin{center}
 \begin{tikzcd}
   R_0 \arrow[dr] \arrow[r, shorten >=1ex,shorten <=1ex]
   & R_1 \arrow[d] 
   \arrow[l, shift left=1.2] \arrow[l, shift right=1.2] 
   \arrow[r, shift right, shorten >=1ex,shorten <=1ex ] \arrow[r, shift left, shorten >=1ex,shorten <=1ex] 
   & R_2 \arrow[dl] 
   \arrow[l] \arrow[l, shift left=2] \arrow[l, shift right=2] 
   \arrow[r, shorten >=1ex,shorten <=1ex] \arrow[r, shift left=2, shorten >=1ex,shorten <=1ex] \arrow[r, shift right=2, shorten >=1ex,shorten <=1ex]
   & \cdots 
   \arrow[l, shift right=1] \arrow[l, shift left=1] \arrow[l, shift right=3] \arrow[l, shift left=3] 
   \\
   & W & 
 \end{tikzcd}
 ,
\end{center}
recovers functors valued in $(\infty,1)$-categories. This intuition is in fact correct and the goal of this paper 
is to make this statement precise and prove it.

Combining our previous observation, we are thus studying two ways of constructing Cartesian fibrations, the marked approach 
and the complete Segal approach, using two models of $(\infty,1)$-categories, quasi-categories and complete Segal spaces, 
giving us a total of four characterizations of Cartesian fibrations, which we can depict as follows:
\begin{center}
 \begin{tabular}{|c|c|c|}
  \hline
  & Marked Approach & Complete Segal Object Approach \\ \hline 
  Quasi-Categories  & \somethingelse{(\sset^+_{/S})^{Cart}}{\text{\cref{the:cart model structure marked simp set}}} & 
  \somethingelse{(\ssset_{/S})^{Cart}}{\text{\cref{the:cart model structure bisimp set}}} \\ \hline 
  Complete Segal Spaces & \somethingelse{(\ss^+_{/X})^{Cart}}{\text{\cref{the:cart model structure marked simp spaces}}} & 
  \somethingelse{(\sss_{/X})^{Cart}}{\text{\cref{the:cart model structure bisimp spaces}}} \\ \hline 
 \end{tabular}
\end{center}
 The marked approach using quasi-categories is the original approach due to Lurie. 
 The complete Segal approach using complete Segal spaces is studied on its own in \cite{rasekh2021cartesiancss}, and directly results in the complete Segal 
 approach using quasi-categories, which we cover in \cref{subsec:cart fib bisimp}. 
 This leaves us with the marked approach using complete Segal spaces, which we study in
 \cref{sec:marked simplicial spaces and Cartesian fibrations}. 
 Finally, we show all four model structures are indeed Quillen equivalent in \cref{sec:equivalence of Cartesian model structures}.

\subsection{Why the Complete Segal Object Approach?}
 Given that we already have a working definition of Cartesian fibrations using marked simplicial objects, 
 why present an alternative approach using complete Segal objects?
 
 {\bf Studying Cartesian Fibrations via Right Fibrations:}
 One set of applications follows from the 
 key observation that in the complete Segal approach Cartesian fibrations are a direct generalization of 
 right fibrations and thus, their properties can be directly deduced from right fibrations.
 \begin{itemize}
  \item {\bf Generalized Cartesian Fibrations:}
   We can relax the complete Segal conditions used to define Cartesian fibrations to define generalized Cartesian fibrations, that can be used to study presheaves valued in various other localizations of simplicial spaces, such as {\it Segal spaces} \cite{rasekh2021cartesiancss}. 
  \item {\bf Invariance of Cartesian Fibrations:}
  One important result about the Cartesian model structure is the fact that it is {\it invariant under categorical equivalences}, 
  meaning base change by a categorical equivalence gives us a Quillen equivalence. 
  This is proven in \cite[Proposition 3.3.1.1]{lurie2009htt} using the fact that the Cartesian model structure 
  is Quillen equivalent to a presheaf category.
  
  We can use a similar equivalence between the contravariant model structure and a presheaf category to similarly deduce that 
  right fibrations are invariant under categorical equivalences. However, in this case there is also an alternative proof, which 
  uses the combinatorics of the contravariant model structure \cite{heutsmoerdijk2015leftfibrationi}.
  Using the complete Segal object approach to Cartesian fibrations we can generalize that proof immediately from right fibrations 
  to Cartesian fibrations without having to translate to presheaf categories. 
  \item {\bf Grothendieck Construction of Cartesian Fibrations:} 
  Another example where we can exploit the connection to right fibrations is the Grothendieck construction. 
  In \cite{lurie2009htt}, Lurie first proves the equivalence between right fibrations and space-valued presheaves. 
  He then wants to generalize that to an equivalence between Cartesian fibrations and $(\infty,1)$-category-valued presheaves,
  but cannot directly do so and thus needs several detailed and complicated proofs.
  
  On the other hand, in the complete Segal object approach to Cartesian fibrations every proof of the equivalence between 
  right fibrations and space valued presheaves immediately generalizes to a new proof of the Grothendieck fibration for 
  Cartesian fibrations.
  
  \item {\bf Grothendieck Construction over $1$-Categories:}
  One particular instance of the previous point is the study of Cartesian fibrations over ordinary categories. 
  In \cite{heutsmoerdijk2015leftfibrationi}, the authors give a Grothendieck construction for right fibrations very much along the lines of the
  classical Grothendieck construction for categories (also known as the {\it category of elements}). They then prove that this gives us a Quillen equivalence. 
  Again, using the complete Segal approach, we can realize that using the category of elements we can construct an equivalence 
  between Cartesian fibrations over categories and presheaves. This construction is much simpler than the general 
  unstraightening construction and even simpler than the specific one given over categories in \cite[3.2.5]{lurie2009htt}.
 \end{itemize}

 {\bf Studying Fibrations of $(\infty,n)$-Categories:} 
 The way we generalized $1$-categories to $(\infty,1)$-categories, we can generalize strict $n$-categories to $(\infty,n)$-categories 
 \cite{bergner2020surveyn}. 
 Moreover, similar to the $(\infty,1)$-categorical case, the study of functors 
  $$ P: \C^{op} \to \cat_{(\infty,n)},$$
 where $\C$ is an $(\infty,n)$-category, is quite a challenge, as we encounter the same problems with homotopy coherence. 
 Thus we again would like to develop an appropriate theory of fibrations. 
 
 However, the $(\infty,n)$-categorical case faces far more challenges than the $(\infty,1)$-categorical case. 
 First of all, similar to the $(\infty,1)$-categories, there are several models of $(\infty,n)$-categories: 
 {\it $\Theta_n$-spaces} \cite{rezk2010thetanspaces},
 {\it $\Theta_n$-sets} \cite{ara2014highersegal},
 {\it $n$-fold complete Segal spaces} \cite{barwick2005nfoldsegalspaces},
 {\it complicial sets} \cite{verity2008complicial},
 {\it comical sets} \cite{campionkapulkinmaehara2020comical},
 ... . However, it is not known whether these models are in fact equivalent (except for $\Theta_n$-sets, $\Theta_n$-spaces and
 $n$-fold complete Segal spaces \cite{bergnerrezk2013comparisoni,bergnerrezk2020comparisonii,ara2014highersegal}).
 Second, we don't have a general definition of fibrations for any of these models of $(\infty,n)$-categories.\footnote{There are some results about fibrations of $2$-complicial sets \cite{lurie2009goodwillie,gagnaharpazlanari2020inftytwolimits},  there called {\it scaled simplicial sets}.}
 
 Hence, unlike for $(\infty,1)$-categories, we cannot just focus on one model, prove all the relevant results, 
 and then translate the results, via a web of Quillen equivalences. If we need a theory of fibrations for a certain model, 
 then we really need to study it on its noted. 
 
 Finally, we also cannot just ignore certain models, with the hope that future results will facilitate translating results 
 from one model to another. First of all it is not certain that such equivalences of models will be developed anytime soon 
 (even the study of strict $n$-categories is quite challenging) and second of all most models already have concrete applications 
 in other branches of mathematics. 
 For example $n$-fold complete Segal spaces are currently the primary method for the study of topological field theories 
 \cite{lurie2009cobordism,calaquescheimbauer2019cobordism}. Or alternatively $2$-complicial sets (or alternatively $2$-comical sets) have found applications in the study of derived algebraic geometry \cite{gaitsgoryrozenblyum2017dagI,gaitsgoryrozenblyum2017dagII}, as explained in \cite[Page 2]{campionkapulkinmaehara2020comical}.
 
  As $\Theta_n$-spaces and $n$-fold complete Segal spaces are directly generalizations 
 of complete Segal spaces, we would expect that a complete Segal approach towards Cartesian fibrations can be a powerful tool 
 in the study of their fibrations.
 
\subsection{Where to go from here?}
 There are several interesting question about the connection between the various Cartesian model structures that remains unexplored:
 \begin{enumerate}
  \item The equivalence between the Cartesian model structures relies on proving that all four are equivalent to various 
  functor categories. It is not clear whether we can construct a direct equivalence between marked simplicial objects and bisimplicial objects 
  that induces an equivalence between the associated Cartesian model structures.
  \item Along the same lines, one of the benefits of using the marked simplicial approach is that it is reasonably easy to see that 
  Cartesian fibrations are themselves categorical fibrations. We would like to generalize this result and, for example, prove that 
  generalized Cartesian fibrations that characterize functors valued in Segal spaces are always Segal fibrations \cite{rasekh2021cartesiancss}. 
  However, the current Quillen equivalences do not allow us to draw such a conclusion.
  
  Answering this question might rely on first having a better way to translate between marked simplicial objects and bisimplicial objects. 
  \item 
  Assuming we can construct a direct equivalence between marked simplicial sets and bisimplicial sets, a final question would be 
  whether we can effectively use that to deduce all the results about Cartesian fibrations as originally proven in 
  \cite[Chapter 3]{lurie2009htt}, hence significantly simplifying the results.
 \end{enumerate}

\subsection{Relation to Other Work} 
 This paper is the second part of a three-paper series which introduces the bisimplicial approach to Cartesian fibrations: 
 \begin{enumerate}
  \item {\bf Cartesian Fibrations of Complete Segal Spaces} \cite{rasekh2021cartesiancss}
  \item {\bf Quasi-Categories vs. Segal Spaces: Cartesian Edition}
  \item {\bf Cartesian Fibrations and Representability} \cite{rasekh2017cartesianrep}
 \end{enumerate}
 In particular, the first paper includes a detailed analysis of the Cartesian model structure on bisimplicial spaces,
 which we only review here (\cref{subsec:cart fib bisimp}). 
 The third paper gives an application of the bisimplicial approach to the study of representable Cartesian fibrations.

\subsection{Notations} \label{subsec:notation}
 Given that the goal is to prove existence of several Quillen equivalences, we use many different model structures 
 (some times on the same category). Thus, in order to avoid any confusion, we will denote every category with its associated model structure.
 In order to help the reader, here is a list of all relevant model structures (except the four 
 Cartesian model structures already mentioned), the underlying category, along with the abbreviation and a reference 
 to their definition:
 
 \begin{center}
  \begin{tabular}{|l|l|l|l|}
   \hline 
   {\bf Model Structure} & {\bf Category(ies)} & {\bf Abbreviation} & {\bf Reference} \\ \hline 
   Joyal Model Structure & $\sset$ & $Joy$ & \cref{the:joyal model structure} \\ \hline 
   Complete Segal Space Model Structure & $\ss$ & $CSS$ & \cref{the:css model structure}  \\ \hline 
   Complete Segal Object Model Structure & $s\M$ & $CSO$ & \cref{the:cso model structure} \\ \hline 
   Contravariant Model Structure & $\ss_{/X}$ & $contra$ & \cref{the:contravariant Model Structure ss} \\ \hline 
   Contravariant Model Structure & $\sset_{/S}$ & $contra$ & \cref{the:contravariant Model Structure sset} \\ \hline 
      localized unmarked Reedy Model Structure & $\ss^+$ & $un^+Ree_\mathcal{L}$& 
   \cref{prop:transferred model structure on marked simplicial spaces} \\ \hline 
      unmarked CSS Model Structure & $\ss^+$ & $un^+CSS$ & \cref{cor:marked CSS} \\ \hline  
      marked CSS Model Structure & $\ss^+$ & $CSS^+$ & \cref{the:absolute Cartesian model structure on marked simplicial spaces} \\ \hline 
  \end{tabular}
 \end{center}
 
\subsection{Acknowledgments} \label{Subsec Acknowledgements}
 I would like to thank Charles Rezk for suggesting the study of Cartesian fibrations of simplicial spaces, 
 which has been the main motivation for this work. I would also like to thank Joyal and Tierney for their beautiful work in \cite{joyaltierney2007qcatvssegal}, which is the theoretical backbone of this paper and has been the motivation for the title. 

\section{Review of Relevant Concepts} \label{Sec Reminders Notation}
 In this section we review some important concepts and establish necessary notation. 
  
 \subsection{A Plethora of Simplicial Objects} \label{subsec:simp object}
 Let $\Delta$ be the simplex category with object posets $[n] =  \{ 0,1,...,n \} $ and morphisms maps of posets. 
 We will use a variety of simplicial objects, i.e. functors $X: \Delta^{op} \to \C$ 
 and thus need to distinguish them carefully, based on the value of the simplicial objects.
 
 {\bf Simplicial Sets/Spaces:} We will use two terminologies for functors $X: \Delta^{op} \to \set$: 
 \begin{itemize}
  \item It is a {\it space} if it is an object in the category of simplicial sets with the Kan model structure. 
  In that case the category of spaces is denote $\s$. 
  \item It is a {\it simplicial set} if it is an object in the category of simplicial sets with any other model structure 
  (such as Joyal, contravariant, ...). In that case the category of simplicial sets is denoted by $\sset$.
 \end{itemize}

 We will use following notation regarding simplicial sets/spaces:
 
 \begin{enumerate} 
  \item \label{item:morphism notation} $\Delta[n]$ denotes the simplicial set representing $[n]$ i.e. $\Delta[n]_k = \Hom_{\Delta}([k], [n])$. 
  \item $\partial \Delta[n]$ denotes the boundary of $\Delta[n]$ i.e. 
  the largest sub-simplicial set which does not include $id_{[n]}: [n] \to [n]$.
  Similarly $\Lambda[n]_l$ denotes the largest simplicial set in $\Delta[n]$ which does not include $l^{th}$ face.
  \item \label{item:Jn} Let $I[l]$ be the category with $l$ objects and one unique isomorphisms between any two objects.
  Then we denote the nerve of $I[l]$ as $J[l]$. 
  It is a Kan fibrant replacement of $\Delta[l]$ and comes with an inclusion $\Delta[l] \rightarrowtail  J[l]$,
  which is a Kan equivalence.
 \end{enumerate}
 
 {\bf Bisimplicial Sets/Simplicial Spaces:}
 Similar to above we will use two terminologies for functors $X: \Delta^{op} \times \Delta^{op} \to \set$: 
 \begin{itemize}
  \item It is a {\it simplicial space} if it is an object in the category of bisimplicial sets with the Reedy model structure 
  or any localization thereof. 
  In that case the category of simplicial spaces is denote $\ss$. 
  \item It is a {\it bisimplicial set} if it is an object in the category of bisimplicial sets with any other model structure 
  (such as Reedy contravariant, ...). In that case the category of bisimplicial sets is denoted by $\ssset$.
 \end{itemize}
 
  We will use following notation convention for simplicial spaces.
  
  \begin{enumerate}
  \item \label{item:Fn} Denote $F(n)$ to be the simplicial space defined as
  $$F(n)_{kl} = \Hom_{\Delta}([k],[n]).$$
  Similarly, we define the simplicial space $\Delta[l]$ as
  $$\Delta[n]_{kl} = \Hom_{\Delta}([l,[n])$$
  Note, $F(n) \times \Delta[n]$ generate the category of simplicial spaces. 
  \item Let $E(n)$ be denote the simplicial space defined as $E(n)_{kl}= J[n]_k$. In particular $E(1)$ is also known as the {\it walking isomorphism}.
  \item \label{item:embed} 
  We embed the category of spaces inside the category of simplicial spaces as constant simplicial spaces 
  (i.e. the simplicial spaces $S$ such that $S_n = S_0$ for all $n$ and all simplicial operator maps are identities). 
  \item $\partial F(n)$ denotes the boundary of $F(n)$. Similarly $L(n)_l$ 
  denotes the largest simplicial space in $F(n)$ which lacks the $l^{th}$ face.
  \item The category $\ss$ is enriched over spaces
  $$\Map_{\ss}(X,Y)_n = \Hom_{\ss}(X \times \Delta[n], Y).$$
  \item The category $\ss$ is also enriched over itself: 
  $$(Y^X)_{kn} = \Hom_{\ss}(X \times F(k) \times \Delta[n], Y).$$
  \item By the Yoneda lemma, for a simplicial space $X$ we have a bijection of spaces
  $$X_n \cong \Map_{\ss}(F(n),X).$$
  \end{enumerate}
 
 {\bf Bisimplicial Spaces:}
 We also use functors $X:\Delta^{op} \times \Delta^{op} \times \Delta^{op} \to \set$, 
 which we call {\it bisimplicial space} and denote by $\sss$. 

  \subsection{Diagrammatic Model Structures} \label{Subsec Reedy Model Structure}
  In this section we review some of the underlying model structures that we can define on the categories defined in \cref{subsec:simp object}.
  
  {\bf Kan Model Structure:}
  The {\it Kan model structure} is one of the first model structures defined and can already be found in \cite{quillen1967modelcats}.
  For another detailed account of the Kan model structure see \cite{goerssjardine1999simplicialhomotopytheory}.
  
  \begin{theone} \label{the:kan model structure}
   There is a unique simplicial, combinatorial, proper model structure on $\s$, called the Kan model structure and denoted $\s^{Kan}$ 
   characterized as follows:
   \begin{enumerate}
    \item A map is a cofibration if it is a monomorphism.
    \item A map is a fibration if satisfies the lifting right lifting property with respect to horns $\Lambda[n]_i \to \Delta[n]$.
    \item A map is a weak equivalence if it is homotopy equivalence. (its geometric realization is an equivalence of topological spaces).
   \end{enumerate}
  \end{theone}
  
  {\bf Reedy Model Structure for Simplicial Objects:}
  Let $\C^\M$ be a model category with model structure $\M$. 
  Then we can define a new model structure on the category of simplicial objects $\Fun(\Delta^{op}, \C) = s\C$.
  This model structure exists for a wide range of model categories, however, here we will focus on the case of interest to us.
  It was originally constructed by Reedy \cite{reedy1974modelstructure}, however we will 
  use \cite{hirschhorn2003modelcategories} as our primary reference.

  \begin{theone} \label{the:reedy model structure}
  \cite[Theorem 15.3.4, Proposition 15.6.3]{hirschhorn2003modelcategories}
   Let $\C^\M$ be a combinatorial, simplicial, left proper model structure such that the cofibrations are monomorphisms. 
   Then there exists a simplicial, combinatorial, left proper model structure on simplicial objects in $\C$, 
   which we call the {\it Reedy model structure} and denote $s\C^{Ree_\M}$,
   which has following characteristics:
   \begin{enumerate}
    \item A map $p: Y \to X$ is a Reedy cofibration if it is a monomorphism.
    \item A map $p:Y \to X$ is a Reedy weak equivalence if it is level-wise a weak equivalence in $\C^\M$.
    \item A map $p:Y \to X$ is a Reedy fibration if for every $n \geq 0$ the induced map 
    $$Y_n \to \mathrm{M}_nY \times_{\mathrm{M}_nX} X_n$$
    is a fibration in $\C^\M$. Here $\mathrm{M}_n$ is the {\it $n$-th matching object} \cite[Definition 15.2.5]{hirschhorn2003modelcategories}.
    \end{enumerate}
  \end{theone}
  
  \begin{remone} \label{rem:reedy fibrant in simplicial spaces}
   If $\C^\M = \s^{Kan}$, then the $n$-th matching object of a simplicial space $X$ 
   is given by the space $\mathrm{M}_nX= \Map_{\ss}(\partial F(n), X)$ \cite[Proposition 15.6.15]{hirschhorn2003modelcategories}. 
  \end{remone}

  Notice the construction is invariant under Quillen equivalences.
  
  \begin{lemone} \label{lemma:reedy invariant}
  \cite[Proposition 15.4.1]{hirschhorn2003modelcategories}
   Let 
   \begin{center}
    \adjun{\C^\M}{\D^\N}{F}{G}
   \end{center}
    be a Quillen equivalence between model categories. 
    Then the induced adjunction 
   \begin{center}
    \adjun{(s\C)^{Ree_\M}}{(s\D)^{Ree_\N}}{sF}{sG}
   \end{center}
   is a Quillen equivalence as well.
  \end{lemone}

  {\bf Projective Model Structure for Functor Categories}
  The existence of the Reedy model structure relied on the fact that $\Delta$ is a Reedy category and so the Reedy model 
  structure cannot be applied to any functor category $\Fun(\C,\D)$. Fortunately, under mild assumptions on a model category, we 
  can construct the {\it projective model structure}.
  
  \begin{theone} \label{the:projective model structure}
   \cite[Theorem 11.6.1, Theorem 11.7.3]{hirschhorn2003modelcategories}
   Let $\C$ be a small category and $\D^\M$ a combinatorial, simplicial, left proper model category. 
   Then there exists a unique left proper, simplicial, combinatorial model structure on the functor category, 
   denoted $\Fun(\C,\D^\M)^{proj}$ and called the {\it projective model structure}, such that 
   \begin{enumerate}
    \item A map $\alpha: F \to G$ is a projective fibration if for all objects $c$ in $\C$ the map 
    $\alpha_c: F(c) \to G(c)$ is a fibration in $\D^\M$. 
    \item A map $\alpha: F \to G$ is a projective weak equivalence if for all objects $c$ the map
    $\alpha_c: F(c) \to G(c)$ is a weak equivalence in $\D^\M$. 
   \end{enumerate}
  \end{theone}
      
  Similar to the Reedy model structure the projective model structure is also invariant.
  
  \begin{lemone} \label{lemma:projective model invariant}
   \cite[Theorem 11.6.5]{hirschhorn2003modelcategories}
     Let 
   \begin{center}
    \adjun{\D^\M}{\E^\N}{F}{G}
   \end{center}
    be a Quillen equivalence between combinatorial simplicial left proper model categories and let $\C$ be a small category. 
    Then the induced adjunction 
   \begin{center}
    \adjun{\Fun(\C,\D^\M)^{proj}}{\Fun(\C,\E^\N)^{proj}}{F_*}{G_*}
   \end{center}
   is a Quillen equivalence as well.
  \end{lemone}

 \subsection{Constructing New Model Structures} 
 In this subsection we review several important (and technical) results that help us construct new model structures.
 
 One very common technique is known as Bousfield localization.
 
 \begin{theone} \label{the:bousfield localization}
  Let $\C^\M$ be a category with a left proper, combinatorial, simplicial model structure $\M$.
  Moreover, let $\mathcal{L}$ be a set of cofibrations in $\C^\M$. 
  There exists a unique cofibrantly generated, left proper, simplicial model category structure on $\C$, 
  called the $\L$-localized model structure and denoted $\C^{\M_\L}$ with the following properties:
  \begin{enumerate}
   \item A morphism $Y \to Z$ is a cofibration in $\C^{\M_\L}$ if it is a cofibration in $\C^\M$.
   \item An object $W$ is fibrant (also called $\L$-local) if it is fibrant in $\C^\M$ and
   for every morphism $f: A \to B$ in $\L$, the induced map  
   $$\Map_\C(B, W) \to \Map_\C(A, W)$$
   is a Kan equivalence.
   \item A map $Y \to Z$ is a weak equivalence in $\M_\L$ if for every $\L$-local object $W$ the induced map 
   $$\Map_\C(Z, W) \to \Map_\C(Y, W)$$
   is a Kan equivalence.
   \end{enumerate}
  \end{theone}
 
 The existence of the $\L$-localized model structure is proven using the theory of left Bousfield localizations. 
 
 For a careful and detailed proof of the existence of left Bousfield localizations see \cite[Theorem 4.1.1]{hirschhorn2003modelcategories}
 For a nice summary of this proof that goes over the main steps see \cite[Proposition 9.1]{rezk2001css}.
 
 Using a variation of Bousfield localization and ideas from \cite{quillen1967modelcats}  get right induced model structures.
 
 \begin{theone} \label{the:right induced}
  Let 
  \begin{center}
   \adjun{\D^\M}{\C}{F}{G}
  \end{center}
  be an adjunction, such that the following conditions hold:
  \begin{enumerate}
   \item $\C$ is a simplicial, locally presentable category.
   \item $\D^\M$ is a simplicial, left proper, combinatorial model category
   \item $G$ preserves filtered colimits
   \item $\C$ has a functorial fibrant replacement, meaning there is a functor $R: \C \to \C$ and natural transformation $X \to RX$, such that $GX \to GRX$ is a fibrant replacement in $\D$.
  \end{enumerate}
 Then there exists a unique simplicial, combinatorial, left proper model structure on $\C^\N$, called the right induced model structure,
 and characterized by the fact that a 
 map $f: c \to d$ in $\C^\N$ is a fibration (weak equivalence) if and only if $Gf$ is a fibration (weak equivalence) in $\D^\M$.
 \end{theone}

 The proof is a combination of results \cite[Theorem 3.6, Theorem 3.8, Corollary 4.14]{goerssschemmerhorn2007modelcats}.
 
 There is also another version of a Bousfield localization we need.
 
 \begin{theone} \label{the:localizing adjunctions}
  \cite[A.3.7.10]{lurie2009htt}
  Let $\C^\M$ and $\D^\N$ be left proper combinatorial simplicial
  model categories such that cofibrations on in $\C^\M$ and $\D^\N$ are monomorphisms 
  and suppose we are given a simplicial Quillen adjunction
 \begin{center}
  \adjun{\C^\M}{\D^\N}{F}{G}
 \end{center}
 Then 
 \begin{enumerate}
  \item There exists a new left proper combinatorial simplicial model structure $\C^{\M'}$
  on the category $\C$ with the following properties: 
  \begin{itemize}
   \item A morphism $f$ in $\C^{\M'}$
   is a cofibration if and only if it is a cofibration in $\C^{\M}$.
   \item A morphism $f$ in $\C^{\M'}$
   is a weak equivalence if and only if $Ff$ is a weak equivalence in $\D^\N$.
   \item A morphism in $f$ in $\C^{\M'}$
   is a fibration if and only if it has the right
   lifting property with respect to trivial cofibrations.
  \end{itemize}
  \item The functors F and G determine a new simplicial Quillen adjunction 
  \begin{center}
  \adjun{\C^{\M'}}{\D^\N}{F}{G}
 \end{center}
  \item If the the derived right adjoint is fully faithful then the Quillen adjunction is a Quillen equivalence.
 \end{enumerate}
\end{theone}
 
 Finally, we want to observe how a Quillen equivalence of model categories can be transferred to over-categories.
 
 \begin{propone} \label{prop:quillen equiv over cat}
  Let $\C^\M$ and $\D^\N$ be two model categories in which all objects are cofibrant and let
  \begin{center}
   \adjun{\C^\M}{\D^\N}{F}{G}
  \end{center}
  be a Quillen adjunction. Fix an object $C \in \C$. Then the adjunction
    \begin{center}
   \adjun{(\C_{/C})^\M}{(\D_{/FC})^\N}{F}{u^*G}
  \end{center}
  is a Quillen adjunction. Here
  $F(C' \to C) = FC' \to FC$ and $u^*G(D \to FC)= (u_C)^*(GD \to GFC)$ where $u_C: C \to GFC$ is the unit map of the adjunction.
  Moreover, if $(F,G)$ is simplicial (or a Quillen equivalence) then $(F,u^*G)$ is also simplicial (or a Quillen equivalence).
 \end{propone}

 \subsection{Model Structures for $(\infty,1)$-Categories and their Equivalence}
 There are now many different approaches to the theory of $(\infty,1)$-categories. 
 Here we focus on two very popular models. Quasi-categories and complete Segal spaces. 
 
 Quasi-categories were first introduced by Boardman and Vogt in their study of homotopy coherent algebraic structures 
 \cite{boardmanvogt1973qcats}. It was Joyal who realized that quasi-categories can be used to do concrete category theory
 \cite{joyal2008notes,joyal2008theory} and then later Lurie, who developed all of category theory in the 
 context of quasi-categories \cite{lurie2009htt,lurie2017ha}.
 Here we only review the definition and the existence of a model structure.
 
 \begin{defone} 
 \cite[Definition 1.1.2.4]{lurie2009htt}
  A {\it quasi-category} is a simplicial set $S$ that has right lifting property with respect to diagrams 
  \begin{center}
   \begin{tikzcd}
    \Lambda[n]_i \arrow[r] \arrow[d] & S \\
    \Delta[n] \arrow[ur, dashed] 
   \end{tikzcd}
  \end{center}
  where $0 < i < n$. 
 \end{defone}
 Quasi-categories have many characteristics of a category. In particular, the set $S_0$ is the set of objects, $S_1$ are the morphisms. Moreover, $S_{hoequiv} \subseteq S_1$ is the subset of $S_1$ consisting of the equivalences in $S_1$.
 Quasi-categories come with a model structure.
 
 \begin{theone} \label{the:joyal model structure}
 \cite[Theorem 2.2.5.1]{lurie2009htt}
  There is a unique, combinatorial, left proper model structure on $\sset$, called the {\it Joyal model structure} and denoted 
  $\sset^{Joy}$, with the following specifications:
  \begin{itemize}
   \item A map $S \to T$ is a cofibration if it is a monomorphism.
   \item An object $S$ is fibrant if it is a quasi-category.
  \end{itemize}
 \end{theone}

 Another popular model of $(\infty,1)$-categories are complete Segal spaces, which were defined and 
 studied by Charles Rezk \cite{rezk2001css}.

  \begin{defone} \label{def:segal and completeness maps}
   Let $G(n)$ be the sub-simplicial space of $F(n)$ consisting of maps $f: [k] \to [n]$ such that $f(k) - f(0) \leq 1$.
   Define the sets $\Seg$ and $\Comp$ as follows:
   $$\Seg = \{ G(n) \hookrightarrow F(n) : n \geq 2 \}$$
   $$\Comp = \{ F(0) \to E(1) \} $$
  \end{defone}
  
  \begin{notone} \label{not:spine}
   We denote the subobject of $\Delta[n]$ that is defined similarly  to $G(n)$ by $\Sp[n]$. 
  \end{notone}
  
  \begin{defone} \label{def:complete Segal Spaces}
  \cite[6]{rezk2001css}
   A Reedy fibrant simplicial space $W$ is called a {\it complete Segal space} if for every map $f: A \to B$ in $\Seg \cup \Comp$, 
   the induced map 
   $$\Map_{\ss}(B,W) \to \Map_{\ss}(A,W)$$
   is a Kan equivalence.
  \end{defone}
  
  Similar to quasi-categories, complete Segal spaces have the characteristics of a category. In particular, the elements in $W_{00}$ are the objects and the elements in $W_{10}$. Moreover, the denote the subspace of $W_1$ consisting of equivalence by $W_{hoequiv} \subseteq W_1$ (following \cite[Section 6]{rezk2001css}.
  The definition already suggests that complete Segal spaces are fibrant objects in a model structure.
 
 \begin{theone} \label{the:css model structure}
  \cite[Theorem 7.2]{rezk2001css}
  There is a unique simplicial combinatorial, left proper model structure on the category of simplicial spaces, 
  called the {\it complete Segal space model structure} and denoted by $\ss^{CSS}$, 
  given by the following specifications:
  \begin{enumerate}
   \item A map $f: Y \to Z$ is a cofibration if it is a monomorphism.
   \item An object $W$ is fibrant if it is a complete Segal space.
   \item A map $f: Y \to Z$ is a weak equivalence if for every complete Segal space $W$ the induced map 
   $$\Map_{\ss}(Z,W) \to \Map_{\ss}(Y,W)$$
   is a Kan equivalence.
  \end{enumerate}
 \end{theone}
 
 Before we proceed to the comparison between complete Segal spaces and quasi-categories, we want to generalize complete Segal spaces
 to complete Segal objects. 
 
 Let $\C^\M$ be a model category. The fact that $\C$ has small coproducts implies that we have a functor
 $$\Cop: \set \to \C,$$
 uniquely characterized by the fact that it takes the one element set to the final object in $\C$. 
 This can be extended to a functor of simplicial objects:
 $$\sCop: \sset \to s\C$$
  Now we can use this notation for following definition.
  
  \begin{defone} \label{def:seg and comp maps}
   For a given category $\C$ define the sets $\Seg_\M$ and $\Comp_\M$ as 
   $$ \Seg_\M = \{ \sCop(\Sp[n] \hookrightarrow \Delta[n]) : n \geq 2 \}$$
   $$ \Comp_\M = \{ \sCop(\Delta[0] \hookrightarrow J[1]) \}$$
  \end{defone}

 \begin{defone} \label{def:cso}
  Let $\C^\M$ be a combinatorial, simplicial, left proper model category such that the cofibrations are the monomorphisms. 
  A functor $W: \Delta^{op} \to \C$ is 
  called a {\it complete Segal object} if it satisfies following conditions:
  \begin{enumerate}
   \item It is fibrant in $s\C^{Ree_\M}$.
   \item For every morphism $A \to B$ in $\Seg_\M \cup \Comp_\M$ and every object $K$ in $\C$ the induced map 
   $$\Map_{s\M}(B \times K , W) \to \Map_{s\M}(A \times K,W)$$
   is a Kan equivalence.
  \end{enumerate}
 \end{defone}

 \begin{theone} \label{the:cso model structure}
  Let $\C^\M$ be a combinatorial, simplicial, left proper model category such that the cofibrations are monomorphisms. Then there exists a unique 
  combinatorial, simplicial, left proper model category structure on the category $s\C$, denoted $s\C^{CSO_\M}$ and called the 
  {\it complete Segal object model structure}, such that it satisfies following conditions:
  \begin{enumerate}
   \item A map $A \to B$ is a cofibration if it is a monomorphism of simplicial objects.
   \item An object $W$ is fibrant if it is a complete Segal object.
   \item A map of simplicial objects $Y \to Z$ is a weak equivalence if the induced map 
   $$\Map_{s\C}(Z,W) \to \Map_{s\C}(Y,W)$$
   is a Kan equivalence for every complete Segal object $W$.
  \end{enumerate}
  We will often simplify the notation to $s\C^{CSO}$, if the model structure $\M$ is clear from the context. 
 \end{theone}
 
  \begin{proof}
   By \cref{the:reedy model structure}, $s\C^{Ree_\M}$ is also combinatorial, simplicial and left proper and 
   where the cofibrations are monomorphisms. In particular, $s\C$ is locally presentable and so we can choose a get of generators 
   $\mathcal{I}$
   Hence, by \cref{the:bousfield localization}, we can construct the localized model structure $s\C^{(Ree_\M)_{\L}}$, where 
   $$\L = \{ A \times K \to B \times K | A \to B \in \Seg_\M \cup \Comp_\M, K \in \mathcal{I} \}$$
   which immediately satisfies the desired results.
  \end{proof}

 \begin{remone}
  See \cite[Proposition 2.2.9]{riehlverity2017inftycosmos} 
  for more details about the complete Segal object model structure (there called {\it Rezk object}) and other properties it inherits from 
  $\C^\M$.
 \end{remone}

 Finally, the complete Segal objects are also homotopy invariant.
 
\begin{theone} \label{the:equiv of CSO}
  Let 
  \begin{center}
   \adjun{\C^\M}{\D^\N}{F}{G}
  \end{center}
  be a Quillen equivalence of combinatorial, simplicial, left proper model structures. 
  Then the induced adjunction 
    \begin{center}
   \adjun{(s\C)^{CSO_\M}}{(s\D)^{CSO_\N}}{sF}{sG}
  \end{center}  
  is also a Quillen equivalence. 
 \end{theone}

 \begin{proof}
  This follows directly from combining \cref{lemma:reedy invariant} and \cref{the:localizing adjunctions}.
 \end{proof}
 
 Until now we have claimed that quasi-categories and complete Segal spaces are two models of $(\infty,1)$-categories. 
 This requires knowing that they are in fact equivalent. This was proven by Joyal and Tierney \cite{joyaltierney2007qcatvssegal}, 
 who constructed two Quillen equivalences between complete Segal spaces and quasi-categories that will play an important role 
 later on and hence we will review here.
 
  Let 
  $$p_1: \Delta \times \Delta \to \Delta$$ 
  be the projection functor that takes $([n],[m])$ to $[n]$. 
  
  Similarly, let 
  $$i_1: \Delta \to \Delta \times \Delta$$
  be the inclusion functor that takes $[n]$ to $([n],[0])$.
  
 \begin{theone} \label{the:jt t}
 \cite[Theorem 4.11]{joyaltierney2007qcatvssegal}
  The induced adjunction 
  \begin{center}
   \adjun{\sset^{Joy}}{\ss^{CSS}}{p_1^*}{i_1^*}
  \end{center}
  is a Quillen equivalence between the Joyal model structure and the complete Segal space model structure.
  Here $i_1^*$ and $p_1^*$ are defined by precomposition via the functors $i_1$ and $p_1$.
 \end{theone}
 
 We move on to the second Quillen equivalence. 
 Let 
 $$t: \Delta \times \Delta \to \sset$$
 be the functor defined as $t([n],[m]) = \Delta[n] \times J[m]$. 
 
 \begin{theone} \label{the:jt ip}
 \cite[Theorem 4.12]{joyaltierney2007qcatvssegal}
  Let 
  \begin{center}
   \adjun{\ss^{CSS}}{\sset^{Joy}}{t_!}{t^!}
  \end{center}
  be the adjunction induced by the map $t$ (meaning $t_!$ is the left Kan extension).
  Then this adjunction is a Quillen equivalence between the complete Segal space model structure and the Joyal model structure.
  Here $t_!$ is the left Kan extension along $t$ and $t^!$ is the right adjoint of the left Kan extension. 
 \end{theone}
 
 \begin{remone} \label{rem:t joyal enriched}
  The adjunction $(t_!,t^!)$ is not simplicial adjunction (in fact the Joyal model structure is not even simplicial). 
  However, the adjunction is enriched over the Joyal model structure.
  So, in particular, for two quasi-categories $S,T$ we have an equivalence of complete Segal spaces 
  $$t^!(T^S) \to t^!T^{t^!S}$$
  For a detailed discussion of this enrichment see \cite[Example 2.2.6]{riehlverity2017inftycosmos}
 \end{remone}

 The goal is to show that with some minor adjustments we can generalize these two Quillen equivalences to Quillen equivalences 
 between various notions of Cartesian fibrations. 

\subsection{Contravariant Model Structures} \label{subsec:contravariant model structures}
 In this subsection we define the contravariant model structure for simplicial spaces and simplicial sets 
 and observe that they are in fact Quillen equivalent.
 
 Let us start with the case of simplicial spaces.
 
 \begin{defone} \label{def:right fib simp space}
 \cite[Definition 3.2, Remark 4.24]{rasekh2017left}
 Let $X$ be a simplicial space. A map $p: R \to X$ is called {\it right fibration} if it is a Reedy fibration and the following is a 
 homotopy pullback square:
 \begin{center}
  \pbsq{R_n}{X_n}{R_0}{X_0}{p_n}{\ordered{n}^*}{\ordered{n}^*}{p_0},
 \end{center}
 where $\ordered{n}:F(0) \to F(n)$ is the map that takes the point to $n \in F(n)_0$. 
\end{defone}
 
 Right fibrations come with a model structure.
  \begin{theone} \label{the:contravariant Model Structure ss}
  \cite[Theorem 3.12, Remark 4.24]{rasekh2017left}
  Let $X$ be simplicial space. There is a unique simplicial, combinatorial, left proper 
  model structure on the category $\ss_{/X}$, called the contravariant model structure
  and denoted by $(\ss_{/X})^{contra}$, which satisfies the following conditions:
  \begin{enumerate}
   \item An object $R \to X$ is fibrant if it is a right fibration.
   \item A map $Y \to Z$ over $X$ is a cofibration if it is a monomorphism.
   \item A map $f: Y \to Z$ over $X$ is a weak equivalence if 
   $$\Map_{/X}(Z,R) \to \Map_{/X}(Y,R)$$ 
   is an equivalence for every right fibration $R \to X$.
  \end{enumerate}
 \end{theone}
 
 Now we can move on to the case of simplicial sets. 
 
 \begin{defone} \label{def:right fib simp set}
  \cite[Definition 2.0.0.3]{lurie2009htt}
  A map of simplicial sets $p: R \to S$ is a {\it right fibration} if it has the right lifting property with respect to squares
  \begin{center}
   \liftsq{\Lambda \leb n \reb^i}{R}{\Delta \leb n \reb }{S}{}{}{p}{}
  \end{center}
  where $n \geq 0$ and $ 0  < i \leq n$.
 \end{defone}
 
 Similar to above, we can define a model structure with fibrant objects right fibrations.
 
 \begin{theone} \label{the:contravariant Model Structure sset}
  \cite[Proposition 2.1.4.7, Proposition 2.1.4.8, Proposition 2.1.4.9, Remark 2.1.4.12]{lurie2009htt}
  Let $S$ be a simplicial set. There is a unique simplicial, combinatorial, left proper 
  model structure on the category of simplicial sets over $S$, called the {\it contravariant model structure} and denoted by $(\sset_{/S})^{contra}$, 
  such that 
  \begin{enumerate}
   \item An object $R \to S$ is fibrant if it is a right fibration.
   \item A map $T \to U$ over $S$ is a cofibration if it is a monomorphism.
   \item A map $T \to U$ over $S$ is a weak equivalence if the induced map 
   $$\Map_{/S}(U,R) \to \Map_{/S}(T,R)$$
   is a Kan equivalence for every right fibration $R \to S$.
  \end{enumerate}
 \end{theone}
 
 Notice, we called both model structures (\cref{the:contravariant Model Structure sset}, \cref{the:contravariant Model Structure ss}) 
 contravariant, as they are in fact Quillen equivalent.
 
 \begin{theone} \label{the:equivalence of contra}
 \cite[Theorem B.12, Theorem B.13]{rasekh2017left}
  Let $S$ be a simplicial set and $X$ a simplicial space. There are Quillen equivalences of contravariant model structures
  \begin{center}
   \adjun{(\ss_{/X})^{contra}}{(\sset_{/t_!X})^{contra}}{t_!}{u^*t^!} \ \ 
   \adjun{(\sset_{/S})^{contra}}{(\ss_{/p_1^*S})^{contra}}{p_1^*}{i_1^*}
  \end{center}
 \end{theone}

\subsection{Cartesian Fibrations of Bisimplicial Objects} \label{subsec:cart fib bisimp}
 In this subsection we review the complete Segal approach towards Cartesian fibrations and then prove the analogous 
 results for Cartesian fibrations of bisimplicial sets. 
 The case for bisimplicial spaces has been studied in great detail in \cite{rasekh2021cartesiancss}.
  
 Let $X$ be a simplicial space. We denote by $X$ the bisimplicial space defined as 
 $X_{kn} = X_n$. 
 
 \begin{defone} \label{def:cart fib bisimp spaces}
 \cite[Section 4]{rasekh2021cartesiancss}
  Let $X$ be a simplicial space. A map of bisimplicial spaces  $Y_{\bullet\bullet} \to X$ is a Cartesian fibration if it satisfies following three conditions:
  \begin{enumerate}
   \item $Y \to X$ is a Reedy fibration in $\sss$. 
   \item For each $k$, the map of simplicial spaces $Y_{k\bullet} \to X_\bullet$ is a right fibration.
   \item The map of simplicial spaces  $Y_{\bullet n} \to X_n$ is a complete Segal space fibration. 
  \end{enumerate}
 \end{defone}
 
 \begin{remone} \label{rem:cart fib cso}
  As $X$ is a simplicial space, we have an isomorphism of categories 
   $$\sss_{/X} \cong \Fun(\Delta^{op}, \ss_{/X})$$
  Using this isomorphism we can fact replace condition $(2)$ and $(3)$ with condition 
  \begin{enumerate}
   \item[(3')] The corresponding simplicial functor 
   $$Y: \Delta^{op} \to \ss_{/X}$$
   is a complete Segal object in the contravariant model structure on $\ss_{/X}$. 
  \end{enumerate}
  Indeed, the fact that $Y_k$ is fibrant in the contravariant model structure precisely means that $Y_k$ is a right fibration, which is condition $(2)$. Moreover, $Y \to X$ being a complete Segal object is precisely the statement that $Y_{\bullet n} \to X_n$ is a complete Segal space fibration giving us $(3)$.
 \end{remone}

 \begin{theone} \label{the:cart model structure bisimp spaces}
 \cite[Section 4]{rasekh2021cartesiancss}
  Let $X$ be a simplicial space. 
  There is a unique simplicial combinatorial left proper model structure on $\sss_{/X}$, called the Cartesian model structure
  and denoted $(\sss_{/X})^{Cart}$, 
  such that 
  \begin{itemize}
   \item An object $C \to X$ is fibrant if it is a Cartesian fibration (\cref{def:cart fib bisimp spaces}).
   \item A map $Y \to Z$ over $X$ is a cofibration if it is a monomorphism.
   \item A map $Y \to Z$ over $X$ is a weak equivalence if for every Cartesian fibration $C \to X$, the induced map 
   $$\Map_{/X}(Z,C) \to \Map_{/X}(Y,C)$$
   is a Kan equivalence.
  \end{itemize}
 \end{theone}
 
 \begin{remone} \label{rem:cart model structure cso model structure}
  It follows from \cref{rem:cart fib cso} that this model structure is in fact the complete Segal object model structure on the contravariant 
  model structure on simplicial spaces over $X$, as constructed in \cref{the:cso model structure}.
 \end{remone}

 For more details about the Cartesian model structure (in particular various ways of characterizing its fibrant objects and 
 weak equivalences) see \cite{rasekh2021cartesiancss}.
 
 We now move on to the analogous result for Cartesian fibrations of bisimplicial sets. 
 For a given simplicial set $S$, we denote by $S$ the bisimplicial set defined as 
 $S_{kn} = S_n$. 
 
 \begin{defone} \label{def:cart fib bisimp sets}
 An object $T \to S$ in $\ssset_{/S}$ is a {\it Cartesian fibration} if the corresponding functor 
 $$T: \Delta^{op} \to (\sset_{/S})^{contra}$$
 is a complete Segal object (\cref{def:cso}) in the contravariant model structure on $(\sset_{/S})^{contra}$.
\end{defone}
 
 We want an analogous result about Cartesian fibrations of bisimplicial sets. Here we can simply combine 
 \cref{the:equiv of CSO},
 \cref{the:contravariant Model Structure ss}, 
 and \cref{the:cso model structure} 
 to get following extensive result.
 
 \begin{theone} \label{the:cart model structure bisimp set}
  Let $S$ be a simplicial sets. There exists a unique, simplicial, combinatorial, left proper model structure on the category of 
  bisimplicial sets over $S$, $\ssset_{/S}$, called the Cartesian model structure and denoted $(\ssset_{/S})^{Cart}$
  with following properties:
  \begin{enumerate}
   \item An object $C \to S$ is fibrant if it is a Cartesian fibration (\cref{def:cart fib bisimp sets}).
   \item A map of bisimplicial sets $T \to U$ over $S$ is a cofibration if it is a monomorphism.
   \item A map $T \to U$ over $S$ is a weak equivalence if for every Cartesian fibration $C \to S$, the induced map 
   $$\Map_{/S}(U,C) \to \Map_{/S}(T,C)$$
   is a Kan equivalence.
  \end{enumerate}
  Moreover, we have Quillen equivalences:
 \begin{center}
   \adjun{(\sss_{/X})^{Cart}}{(\ssset_{/st_!X})^{Cart}}{st_!}{u^*st^!} \ \ 
   \adjun{(\ssset_{/S})^{Cart}}{(\sss_{/sp_1^*S})^{Cart}}{sp_1^*}{u^*si_1^*}
   .
  \end{center}
 \end{theone}
 
\section{Marked Simplicial Spaces and Cartesian Fibrations} \label{sec:marked simplicial spaces and Cartesian fibrations}
 In \cref{subsec:cart fib bisimp} we defined two Cartesian model structures: for bisimplicial spaces 
 and for bisimplicial sets.
 As mentioned before, the Cartesian model structure for marked simplicial sets had already been defined \cite{lurie2009htt}.
 This leaves us with one last Cartesian model structure: {\it marked simplicial spaces}, which is the goal of this section.
 
 \subsection{Marked Simplicial Spaces}
  In this subsection we want to study marked simplicial spaces. 
  We first review marked simplicial sets as studied in \cite{lurie2009htt}.
  
  \begin{defone} \label{def:marked simplicial set}
   A {\it marked simplicial set} is a pair $(S,A)$ where $S$ is a simplicial set and $A \subset S_1$ such that $A$ includes 
   all degenerate edges. 
   A morphism of marked simplicial sets $f:(S,A) \to (T,B)$ is a map of simplicial sets $f: S \to T$ such that $f(A) \subset B$.
   We denote the category of marked simplicial sets by $\sset^+$.
  \end{defone}
  
  We want to study the category of marked simplicial sets $\sset_+$. As we have the additional data of markings, the objects 
  are not just simplicial objects. We thus need an alternative approach.
  
  \begin{defone} \label{def:marked simplex}
   Let $\Delta^+$, the {\it marked simplex category}, be the category defined as the pushout
   \begin{center}
    \begin{tikzcd}[row sep=0.5in, column sep=0.5in]
     \{ \leb 1 \reb \to \leb 0 \reb \} \arrow[r, hookrightarrow] \arrow[d, hookrightarrow] & 
     \{ \leb 1 \reb \to \leb 1^+ \reb \to \leb 0 \reb \} \arrow[d] \\
     \Delta \arrow[r] & \Delta^+
    \end{tikzcd}
   \end{center}
   where the top and left hand maps are the evident inclusion maps (note there is a unique map $[1] \to [0]$ in $\Delta$).
  \end{defone}
  
   Using our intuition about $\Delta$, we can depict this category as follows: 
  
  \begin{center}
   \begin{tikzcd}[row sep=0.5in, column sep=0.5in]
   \leb 0 \reb 
   \arrow[r, shift left=1.2] \arrow[r, shift right=1.2]    
   \arrow[dr, shift left=1.2] \arrow[dr, shift right=1.2]   
   & \leb 1 \reb \arrow[d]
   \arrow[l, shorten >=1ex,shorten <=1ex]
   \arrow[r] \arrow[r, shift left=2] \arrow[r, shift right=2] 
   & \leb 2 \reb
   \arrow[l, shift right, shorten >=1ex,shorten <=1ex ] \arrow[l, shift left, shorten >=1ex,shorten <=1ex] 
   \arrow[r, shift right=1] \arrow[r, shift left=1] \arrow[r, shift right=3] \arrow[r, shift left=3] 
   & \cdots 
   \arrow[l, shorten >=1ex,shorten <=1ex] \arrow[l, shift left=2, shorten >=1ex,shorten <=1ex] \arrow[l, shift right=2, shorten >=1ex,shorten <=1ex]
   \\
   & \leb 1^+ \reb \arrow[ul, shorten >=1ex,shorten <=1ex] & 
  \end{tikzcd}
 \end{center}
  where the evident triangle is commutative. 
  
  We would like to prove that a marked simplicial set is just a presheaf on $\Delta^+$,
  where the image of $[1^+]$ is the set of markings. However, for that we need the additional assumption that the 
  image of the map $[1^+] \to [1]$ is an injection of sets. Hence, 
  marked simplicial sets are certain {\it separated} presheaves on $\Delta^+$.

 \begin{lemone} \label{lemma:marked simplicial set equal separated presheaf}
  Let $\J$ be the Grothendieck topology on $\Delta^+$ where the only non-trivial cover is $[1] \to [1^+]$.  
  Then the category of marked simplicial sets is isomorphic to the category of separated set valued presheaves on $\Delta^+$ with the topology $\J$.
 \end{lemone}

 \begin{proof}
  By definition $F: \Delta^+ \to \set$ is separated if it takes covering maps to monomorphisms. 
  Hence, $F$ is separated with respect to the topology $\J$ if and only if $F([1^+]) \to F([1])$ is a monomorphism which means it is a 
  marked simplicial set, where the markings are $F([1^+])$. 
  
  It is now immediate to confirm that a natural transformation of functors corresponds to maps of marked simplicial sets.
 \end{proof}
 
 \begin{remone}
  This lemma is also proven (independently) in \cite{nlab2020marked}.
 \end{remone}
 
  \begin{remone} \label{rem:marked simplicial sets presentable}
    This lemma proves that the category of marked simplicial sets is equivalent to a category of 
    separated presheaves, meaning 
   it is a quasi-Grothendieck topos \cite[C2.2.13]{johnstone2002elephanti}, which has far-reaching implications:
  \begin{enumerate}
   \item $\sset^+$ has small limits and colimits 
   \item $\sset^+$ is a presentable category.
   \item $\sset^+$ is locally Cartesian closed, which for two objects $(X,A)$, $(Y,B)$ we denote by $(Y,B)^{(X,A)}$. 
  \end{enumerate}
 \end{remone}
 
  \begin{notone} \label{not:marked simplicial generators}
   We will make ample use of the fact that marked simplicial objects are subcategories of presheaf categories. 
   Hence we generalize our notational conventions from \cref{subsec:simp object} and
   denote the generators of $\Fun((\Delta^+)^{op},\set)$ by $\Delta^+[n]$. Notice, here $n$ is an object in 
    $\Delta^+$ and so we can also have $n = 1^+$.
  \end{notone}

 Finally, we need to realize how marked simplicial sets are related to simplicial sets \cite[Definition 3.1.0.1]{lurie2009htt}. 
 
  \begin{remone} \label{rem:unmarking sharp flat}
  There is a functor $\inc:\Delta \to \Delta^+$, which is the identity on objects and gives us a chain of adjunctions
  \begin{center}
   \begin{tikzcd}[column sep=1in]
    \Fun((\Delta^+)^{op},\set) \arrow[r, "\inc^*" description]  & 
    \Fun(\Delta^{op}, \set) \arrow[l, bend right=20, "\inc_!"', "\bot"] \arrow[l, bend left=20, "\inc_*", "\bot"']
   \end{tikzcd}
   ,
  \end{center}
  which restricts to adjunctions
  \begin{center}
   \begin{tikzcd}[column sep=1.3in]
    \sset^+ \arrow[r, "\Forget" description]  & \sset \arrow[l, bend right=20, "(-)^\flat"', "\bot"] \arrow[l, bend left=20, "(-)^\#", "\bot"']
   \end{tikzcd}
   .
  \end{center}
  Here $\Forget$ simply forgets the markings and, 
  for a given simplicial set $S$, we have $S^\flat= (S,S_0)$ and $S^\# = (S,S_1)$.
 \end{remone}
 
 \begin{remone} \label{rem:generators marked simplicial sets yoneda}
  We can use the standard simplices $\Delta[n]$ and the functors $(-)^\flat$ and $(-)^\sharp$ to get following valuable 
  isomorphisms:
  $$\Hom_{\sset^+}(\Delta[n]^{\flat}, (S,A)) \cong S_n,$$
  $$\Hom_{\sset^+}(\Delta[1]^\sharp, (S,A)) \cong A,$$
  for all marked simplicial sets $(S,A)$ and $n \geq 0$.
 \end{remone}

   Having done a careful analysis of marked simplicial sets, we are finally ready to study marked simplicial spaces.
  
  \begin{defone} \label{def:marked simp space}
   A {\it marked simplicial space} is a pair $(X,A)$, where $X$ is a simplicial space and $A \hookrightarrow X_1$ is an inclusion of spaces
   such that the degeneracy map 
   $s_0: X_0 \to X_1$ takes value in $A$. We say $X$ is the {\it underlying simplicial space} and $A$ is the {\it space of markings}.
  \end{defone}
   
  \begin{defone} \label{def:category of marked simplicial spaces}
   A map of marked simplicial spaces $f:(X,A) \to (Y,B)$ is a map of simplicial spaces $f: X \to Y$ such that $f_1$ restricts to a map 
   $(f_1)|_A: A \to B$.
   Marked simplicial spaces with the morphisms described before form a category which we denote by $\ss^+$.
  \end{defone}
  
  We now want to proceed to study the category of marked simplicial spaces. 
  Comparing \cref{def:marked simplicial set} and \cref{def:marked simp space} it immediately follows that a 
  marked simplicial space is just a simplicial object in the category of marked simplicial sets.
  Hence, we get the following lemma that allows us to 
  straightforwardly generalize results from marked simplicial sets to marked simplicial spaces.
  
  \begin{lemone}
   There is an equivalence of categories $\Fun(\Delta, \sset^+) \cong \ss^+$.
  \end{lemone}
  
  \begin{remone} \label{rem:marked simplicial spaces remarks}
  The equivalence immediately has following implications:
  \begin{enumerate}
   \item $\ss^+$ has small limits and colimits and presentable.
   \item $\ss^+$ is locally Cartesian closed, which for two objects $(X,A)$, $(Y,B)$ we denote by $(Y,B)^{(X,A)}$. 
   \item Similarly, we denote the generators of $\Fun(\Delta^+,\s)$ by $F^+(n) \times \Delta[l]$.
    Notice, here $n$ is again an object in $\Delta^+$, whereas $l$ is an object in $\Delta$. 
    \item There are adjunctions
  \begin{center}
   \begin{tikzcd}[column sep=1.3in]
    \ss^+ \arrow[r, "\Forget" description]  & \ss \arrow[l, bend right=30, "(-)^\flat"', "\bot"] \arrow[l, bend left=30, "(-)^\#", "\bot"']
   \end{tikzcd}
   .
  \end{center}
  Here $\Forget$ simply forgets the markings and, 
  for a given simplicial space $X$, we have $X^\flat= (X,X_0)$ and $X^\# = (X,X_1)$.
  \item Using $\Forget$ we can define a simplicial enrichment
  $$\Map_{\ss^+}((X,A),(Y,B)) = \Forget((Y,B)^{(X,A)})_0.$$
  \item Finally, we have 
   $$\Map_{\ss^+}(F(n)^{\flat}, (X,A)) \cong X_n,$$
  $$\Map_{\ss^+}(F(1)^\sharp, (X,A)) \cong A,$$
  for all marked simplicial spaces $(X,A)$ and $n \geq 0$.
  \end{enumerate}
  \end{remone}
 
 \subsection{Marked Joyal Model Structure vs. Marked CSS Model Structure}
 In this subsection we want to generalize the Joyal and complete Segal space model structures to marked simplicial sets and spaces. 
 The first step is to adjust the Joyal-Tierney adjunctions (\cref{the:jt t},\cref{the:jt ip}) to marked objects.
 
  Let 
  $$\iplus: \Delta^+ \to \Delta^+ \times \Delta$$
  be the map that takes an object $[n]$ to $([n],[0])$. 
  Similarly, let 
  $$\pplus: \Delta^+ \times \Delta \to \Delta^+$$
  be the functor that takes a pair $([n],[m])$ to $[n]$.

 \begin{defone} \label{def:pplus iplus}
  Let 
  \begin{center}
   \adjun{\Fun(\Delta^+,\set)}{\Fun(\Delta^+ \times \Delta, \set)}{(\pplus)^*}{(\iplus)^*}
  \end{center}
  be the adjunction induced by the maps $\iplus, \pplus$. 
 \end{defone}

 We claim this adjunction restricts to marked objects. 
 
 \begin{lemone} \label{lemma:pplus iplus marked}
  The functors $(\pplus)^*,(\iplus)^*$ take marked objects to marked object. 
  Hence, they restrict to an adjunction 
  \begin{center}
   \adjun{\sset^+}{\ss^+}{(\pplus)^*}{(\iplus)^*}.
  \end{center}
 \end{lemone}
 
 \begin{proof}
  Let $G: (\Delta^+)^{op} \to \set$ be a presheaf that corresponds to a marked simplicial set, meaning the map $G(1^+) \to G(1)$ is a monomorphism. 
  We need to show that $(\pplus)^*(G)(1^+,l) \to (\pplus)^*(G)(1,l)$ is a monomorphism as well. 
  However, by direct computation $(\pplus)^*(G)(1^+,l) \to (\pplus)^*(G)(1,l) = G(1^+) \to G(1)$ and so the result follows from our assumption.
  
  \medskip 
  
  On the other hand, assume $G: (\Delta^+)^{op} \times \Delta^{op} \to \set$ be a presheaf that corresponds to marked simplicial space. 
  We want to prove $(\iplus)^*(G): (\Delta^+)^{op}  \to \set$ corresponds to a marked simplicial set, meaning we have to show 
  $(\iplus)^*(G)(1^+) \to (\iplus)^*(G)(1)$ is a monomorphism. By direct computation, this map is given by $G(1^+,0) \to G(1,0)$, which is a 
  monomorphism by assumption. 
 \end{proof}
 
 We now move on to adjust the second adjunction, $(t_!,t^!)$, to the marked setting. 

 Let 
 $$\tau: \Delta^+ \to \sset^+$$
 be the marked simplicial diagram defined as 
 $$\tau([n]) = 
 \begin{cases}
  \Delta[n]^{\flat} & \text{ if } n \neq 1^+ \\
  \Delta[1]^\sharp & \text{ if } n = 1^+ 
 \end{cases},
 $$
 where the cosimplicial maps between $\Delta[n]^\flat$ are given by the cosimplicial diagram $\Delta$ in $\sset$ and the map 
 $\Delta[1]^\flat \to \Delta[1]^\sharp$ is the evident inclusion map.
 
 We can depict the marked cosimplicial object $\tau$ in $\sset^+$ as
 
  \begin{center}
   \begin{tikzcd}[row sep=0.5in, column sep=0.5in]
   \Delta \leb 0 \reb^\flat 
   \arrow[r, shift left=1.2] \arrow[r, shift right=1.2]    
   \arrow[dr, shift left=1.2] \arrow[dr, shift right=1.2]   
   & \Delta \leb 1 \reb^\flat \arrow[d]
   \arrow[l, shorten >=1ex,shorten <=1ex]
   \arrow[r] \arrow[r, shift left=2] \arrow[r, shift right=2] 
   & \Delta \leb 2 \reb^\flat
   \arrow[l, shift right, shorten >=1ex,shorten <=1ex ] \arrow[l, shift left, shorten >=1ex,shorten <=1ex] 
   \arrow[r, shift right=1] \arrow[r, shift left=1] \arrow[r, shift right=3] \arrow[r, shift left=3] 
   & \cdots 
   \arrow[l, shorten >=1ex,shorten <=1ex] \arrow[l, shift left=2, shorten >=1ex,shorten <=1ex] \arrow[l, shift right=2, shorten >=1ex,shorten <=1ex]
   \\
   & \Delta \leb 1 \reb^\# \arrow[ul, shorten >=1ex,shorten <=1ex] & 
 \end{tikzcd}.
 \end{center}
 
 Using $\tau$ we can now define a functor 
 
 $$\tplus: \Delta^+ \times \Delta \to \sset^+$$
 
 by $\tplus([n],[m]) = \tau(n) \times \Delta[n]^\#$. We can extend this functor to an adjunction.
 
 \begin{defone} \label{def:tplus}
  Let 
  \begin{center}
   \adjun{\Fun((\Delta^+)^{op} \times \Delta^{op}, \set)}{\sset^+}{(\tplus)_!}{(\tplus)^!}
  \end{center} 
  be the adjunction given by left Kan extension along $\tplus$.
 \end{defone}
 
 The right hand side of the adjunction already takes value in marked simplicial sets, so we only need to prove that the 
 right adjoint takes value in marked simplicial spaces as well.
 
 \begin{lemone} \label{lemma:t shriek adjunction}
  The adjunction $((\tplus)_!,(\tplus)^!)$ restricts to an adjunction
    \begin{center}
   \adjun{\ss^+}{\sset^+}{(\tplus)_!}{(\tplus)^!}
  \end{center}
 \end{lemone}
 
 \begin{proof}
  Let $(X,A)$ be a marked simplicial set. Then 
  $$(\tplus)^!(X,A)_{nl}= \Hom_{\sset^+}(\tau(n) \times \Delta[l]^\#, (X,A)) \cong \Hom_{\sset^+}(\tau(n), (X,A)^{\Delta[l]^\#}),$$
  where the isomorphism follows from the fact that $\sset^+$ is Cartesian closed (\cref{rem:marked simplicial sets presentable}).
  
  In order to show that $(\tplus)^!(X,A)$ is a marked simplicial space we have to prove that map 
  $$\Hom_{\sset^+}(\tau(1^+) \times \Delta[l]^\#, (X,A)) \to \Hom_{\sset^+}(\tau(1) \times \Delta[l]^\#, (X,A))$$
  is an injection for all $l \geq 0$.
  By direct computation this map is given by 
  $$\Hom_{\sset^+}(\Delta[1]^\sharp \times \Delta[l]^\#, (X,A)) \to \Hom_{\sset^+}(\Delta[1]^\flat \times \Delta[l]^\#, (X,A)).$$
  By the isomorphism above, this is isomorphic to  
  $$\Hom_{\sset^+}(\Delta[1]^\sharp, (X,A)^{\Delta[l]^\#}) \to \Hom_{\sset^+}(\Delta[1]^\flat, (X,A)^{\Delta[l]^\#}).$$
  Finally, by \cref{rem:generators marked simplicial sets yoneda}, 
  the right hand side is the set of $1$-simplices in $(X,A)^{\Delta[l]^\#}$, whereas the left hand side 
  are the marked ones. Hence, this map is an injection by assumption.
 \end{proof}
 
 \begin{remone} \label{rem:different t definitions}
  Recall that $t: \Delta \times \Delta \to \sset$ was defined as $t([n],[m]) = \Delta[n] \times J[m]$ (\cref{the:jt t}). 
  Hence our first guess for a
  generalization to marked simplicial spaces might have been $t'([n],[m]) = \tau([n]) \times J[m]^\flat$. 
  
  This functor would in fact give us an adjunction between marked simplicial spaces and marked simplicial sets and we could use 
  that to construct Quillen equivalences. 
  However, in \cref{the:marked Joyal} we will observe that the model structure of interest on $\sset^+$, the marked Joyal structure, 
  is in fact a simplicial model structure, and  
  that the simplicial enrichment makes $((\tplus)_!, (\tplus)^!)$ into a simplicial Quillen equivalence 
  (\cref{the:absolute Cartesian model structure on marked simplicial spaces}). 
  If we used $((t')_!,(t')^!)$ instead we would not get a simplicial Quillen equivalence.
  
  On the other hand, we do in fact use the fact that $((\tplus)_!, (\tplus)^!)$ is simplicial, particularly 
  in \cref{the:absolute Cartesian model structure on marked simplicial spaces} and \cref{the:cart model structure marked simp spaces},
  to construct our desired model structures on marked simplicial spaces (we need it to apply \cref{the:localizing adjunctions}).
  Hence we have chosen $\tplus$ as our generalization $t$.
  
  More generally, we note that even without any particular application in mind, having a simplicial adjunction of simplicial model structures implies that we immediately get an adjunction of $\infty$-categories between the underlying $\infty$-categories of the simplicial model structures, as proven in \cite[Proposition 5.2.4.6]{lurie2009htt}, giving additional justification for our choice of generalization.
 \end{remone}
 
 \begin{remone} \label{rem:tplus is enriched}
  Notice the adjunctions $((\tplus)_!,(\tplus)^!)$ and $(t_!,t^!)$ do not give us a commutative square.
  Thus our choice of prioritizing a simplicial Quillen equivalence (as explained in \cref{rem:different t definitions}) comes at the price of not 
  commuting with the original adjunction $(t_!,t^!)$.
 \end{remone}

 We now want to move on and prove that the adjunctions $((\pplus)^*,(\iplus)^*)$ and $((\tplus)_!,(\tplus)^!)$ 
 do in fact give us a Quillen equivalence, similar to the unmarked counter-parts.  
 
 For that we first have to construct appropriate model structures, which we will do in two steps: 
 \begin{enumerate}
  \item First we define a model structure $\ss^+$ which is transferred from a model structure on $\ss$ and 
  does not take the markings into account (\cref{prop:transferred model structure on marked simplicial spaces}).
  \item Then we localize this model structure such that the fibrancy condition depends on an appropriate choice of marking
  (\cref{the:absolute Cartesian model structure on marked simplicial spaces}).
 \end{enumerate}

 \begin{propone} \label{prop:transferred model structure on marked simplicial spaces}
  Let $\mathcal{L}$ be a set of cofibrations of simplicial spaces. 
  There is a unique combinatorial simplicial left proper 
  model structure on the category of marked simplicial spaces, called the {\bf unmarked $\mathcal{L}$-localized model structure} 
  and denoted by 
  $(\ss^+)^{un^+Ree_{\mathcal{L}}}$, with the following specifications:
  \begin{enumerate}
   \item[F] A map $(X,A) \to (Y,B)$ is a fibration if the underlying map of simplicial spaces $X \to Y$ is a 
   fibration in the $\mathcal{L}$ localized Reedy model structure.
   \item[W] A map $(X,A) \to (Y,B)$ is a weak equivalence if the underlying map of simplicial spaces $X \to Y$ is a 
   weak equivalence in the $\mathcal{L}$-localized Reedy model structure.
    \item[C] A map of marked simplicial spaces is a cofibration if it has the left lifting property with respect to 
    maps that are simultaneously fibrations and weak equivalences. 
  \end{enumerate}
  Moreover, the adjunction 
  \begin{center}
   \adjun{\ss^{Ree_{\mathcal{L}}}}{(\ss^+)^{un^+Ree_{\mathcal{L}}}}{(-)^\flat}{\Forget}.
  \end{center}
  is a Quillen equivalence. Here the left hand side has the $\L$-localized model structure (\cref{the:bousfield localization}) and the right side 
  the unmarked $\L$-localized model structure.
 \end{propone}
 
 \begin{proof}
  In order to construct a model structure on $\ss^+$ we want to prove that we can use the adjunction 
  \begin{equation} \label{eq:first equation}
   \adjun{\ss^{Ree_{\mathcal{L}}}}{\ss^+}{(-)^\flat}{\Forget}
  \end{equation}
  to induce the $\mathcal{L}$-localized Reedy model structure via the right adjoint $\Forget$. This would immediately give us most of the 
  desired results, up to including that the adjunction is a Quillen adjunction, leaving us only with the task of showing 
  it is a Quillen equivalence. 
  
  We will hence check the required conditions of right-induced model structures (\cref{the:right induced}):
  
  \begin{itemize}
   \item {\it $\Forget$ commutes with filtered colimits}: This follows from the fact that $\Forget$ is a left adjoint 
   (\cref{rem:marked simplicial spaces remarks}).
   \item {\it $\ss^+$ is simplicial}: Also proven in \cref{rem:marked simplicial spaces remarks}.
   \item {\it There is a functorial fibrant replacement functor in $\ss^+$}: 
   Let $R$ be a fibrant replacement functor in $\ss^{Ree_{\mathcal{L}}}$. Then the functor 
   $$\ss^+ \xrightarrow{ \ \Forget \ } \ss \xrightarrow{ \ R \ } \ss \xrightarrow{ \ (-)^\sharp \ } \ss$$
   is in fact a fibrant replacement functor in the sense of \cref{the:right induced}. Indeed, for a given marked simplicial space $(X,A)$ the map $\Forget(X,A) \to \Forget (R \Forget(X,A)^\sharp)$ is precisely $X \to RX$, which by definition is the fibrant replacement in the $\mathcal{L}$-localized Reedy model structure.
  \end{itemize}
  
  We are now left with proving that the Quillen adjunction (\ref{eq:first equation}) is in fact a Quillen equivalence. 
  The right adjoint reflects weak equivalences by definition and so 
  it suffices to prove the derived unit map is an equivalence for every fibrant simplicial space, however, that is just the identity. 
  Hence we are done.
  \end{proof}  
 
 We can now use this result for particular interesting instances.
 
 \begin{corone} \label{cor:marked Reedy}
  If $\mathcal{L}$ is empty, then we get the unmarked Reedy model structure, $(\ss^+)^{un^+Ree}$, on marked simplicial spaces. 
 \end{corone}
 
 \begin{corone} \label{cor:marked CSS}
  If we take $\mathcal{L} = \Seg \cup \Comp$ (the Segal maps and completeness maps as defined in  \cref{def:segal and completeness maps}), 
  then we get the {\it unmarked complete Segal space model structure}, denoted $(\ss^+)^{un^+CSS}$. 
 \end{corone}
  
   The reason the model structure is called ``unmarked" is that the fibrancy condition is independent of the markings.
  This has very problematic implications, such as the fact that even some common marked simplicial spaces are not cofibrant anymore.
  For example the diagram   
  \begin{center}
   \begin{tikzcd}[row sep=0.5in, column sep=0.5in]
    & F(1)^\flat \arrow[d] \\
    F(1)^\sharp \arrow[ur, dashed, "\nexists" description] \arrow[r, "\id"] & F(1)^\sharp
   \end{tikzcd}
  \end{center}
  has no lift, although right hand map is a trivial fibration in the $\mathcal{L}$-localized marked Reedy model structure. 
  Hence $F(1)^{\sharp}$ is not cofibrant. 
  
  We thus want modify the model structure, by restricting the fibrations and trivial fibrations
  and making it dependent on the markings.
  In order to do that we want to compare the unmarked CSS model structure with the {\it marked Joyal model structure}.
  
  \begin{theone} \label{the:marked Joyal}
   \cite[Proposition 3.1.3.7, Proposition 3.1.4.1, Corollary 3.1.4.4, Proposition 3.1.5.3]{lurie2009htt}
   There is a unique combinatorial, simplicial, left proper model structure on $\sset^+$, called the 
   marked Joyal model structure and denoted $\sset^{Joy^+}$, with the following specifications:
   \begin{enumerate}
    \item A map of marked simplicial sets $f: (S,A) \to (T,B)$ is a cofibration if the underlying map of simplicial sets 
    $S \to T$ is a monomorphism.
   \item  A map of marked simplicial sets $f: (S,A) \to (T,B)$ is a weak equivalence if the underlying map of simplicial sets 
    $S \to T$ is an equivalence in the Joyal model structure.
   \item A marked simplicial set $(X,A)$ is fibrant if $X$ is a quasi-category and $A$ is the set of equivalences in $X$.
   \item The simplicial enrichment is given by the simplicial set 
   $$\Map_{\sset^+}(S,T) = \Hom_{\sset^+}(S \times \Delta[n]^\sharp, T)$$
   \end{enumerate}
   Finally, the adjunction 
   \begin{center}
    \adjun{\sset^{Joy}}{(\sset^+)^{Joy^+}}{(-)^{\flat}}{\Forget}
   \end{center}
   is a Quillen equivalence between the Joyal model structure and the marked Joyal model structure.
  \end{theone}

  \begin{remone}
    This model structure is originally defined in \cite{lurie2009htt} as a special case of the Cartesian model structure 
    for marked simplicial sets (which 
    we will review in \cref{the:cart model structure marked simp set}) and hence does not have its own separate name there.  
  \end{remone}
 
 We want to use the marked Joyal model structure to adjust the unmarked CSS model structure. 
 
 \begin{lemone} \label{lemma:t shriek enriched Quillen equiv}
  The adjunction 
  \begin{center}
   \adjun{(\ss^+)^{un^+CSS}}{(\sset^+)^{Joy^+}}{(\tplus)_!}{(\tplus)^!}
  \end{center}
  is a simplicial Quillen adjunction between the unmarked CSS model structure and the marked Joyal model structure.
 \end{lemone}
 
 \begin{proof}
  First, the adjunction is in fact simplicial. Indeed, by definition we have $(\tplus)_!(\Delta[n]^\flat) = \Delta[n]^\#$, which gives us following natural bijection 
  $$\Map_{\sset^+}((\tplus)_!X,S)_n = \Hom_{\sset^+}(((\tplus)_!X) \times \Delta[n]^\#,S) = \Hom_{\sset^+}((\tplus)_!(X \times \Delta[n]^\flat),S) \cong$$ $$\Hom_{\ss^+}(X \times \Delta[n]^\flat,(\tplus)^!S) = \Map_{\ss^+}(X,(\tplus)^!S)_n$$
  
  As all monomorphisms in $(\sset^+)^{Joy^+}$ are cofibrations, the left adjoint $(\tplus)_!$ indeed preserves cofibrations.
  Hence, in order to prove we have a Quillen adjunction it suffices to prove $(\tplus)^!$ preserves fibrant objects \cite[Corollary A.3.7.2]{lurie2009htt}. However, this is immediate as fibrant objects in marked Joyal model structure are of the form $(S,S_{hoequiv})$ where $S$ is a quasi-category and $\Forget(\tplus)^!(S,S_{hoequiv})=t^!(S)$, which is a complete Segal space, by \cref{the:jt t}, and so $(\tplus)^!(S,S_{hoequiv})$ is fibrant in the unmarked complete Segal space model structure.
\end{proof}  
  Using this Quillen adjunction we can finally construct a new model structure on $\ss^+$, in which the fibrant objects 
  depends on the markings!
  
 \begin{theone} \label{the:absolute Cartesian model structure on marked simplicial spaces}
  There exists a unique simplicial, cofibrantly generated, left proper model structure on 
  marked simplicial spaces, denoted $(\ss^+)^{CSS^+}$ and called the marked CSS model structure, 
  characterized by 
  \begin{enumerate}
   \item Fibrant objects are marked simplicial spaces of the form 
   $(W,W_{hoequiv})$, where $W$ is a complete Segal space and $W_{hoequiv}$ is the subspace of weak equivalences. 
   \item A map $f: (X,A) \to (Y,B)$ is a cofibration if the map of underlying simplicial spaces $X \to Y$ is a monomorphism.
   \item A map $f: (X,A) \to (Y,B)$ is a weak equivalence if for every fibrant object $(W,W_{hoequiv})$ the induced map 
   $$\Map_{\ss^+}((Y,B),(W,W_{hoequiv})) \to \Map_{\ss^+}((X,A),(W,W_{hoequiv}))$$
   is a Kan equivalence.
   \item A map between fibrant objects $(W,W_{hoequiv}) \to (V,V_{hoequiv})$ is 
   a weak equivalence (fibration) if and only if it is a weak equivalence (fibration) in the unmarked CSS model structure.
   \item The adjunction 
    \begin{equation} \label{eq:marked t equiv}
   \adjun{(\ss^+)^{CSS^+}}{(\sset^+)^{Joy^+}}{(\tplus)_!}{(\tplus)^!}
  \end{equation}
   is a simplicial Quillen equivalence of simplicial model structures. 
   \end{enumerate}
 \end{theone}
 
 \begin{proof}
  By \cref{lemma:t shriek enriched Quillen equiv}, we have a simplicial Quillen adjunction
  \begin{center}
   \adjun{(\ss^+)^{un^+CSS}}{(\sset^+)^{Joy^+}}{(\tplus)_!}{(\tplus)^!}.
  \end{center}
  Hence, by applying \cref{the:localizing adjunctions}, we get a new model structure on $\ss^+$, 
  which we call the {\it marked CSS model structure}.
  We want to prove that this model structure satisfies all the conditions stated above. We will start 
  by proving that the adjunction $((\tplus)_!,(\tplus)^!)$ is a Quillen equivalence. 
  
  By \cref{the:localizing adjunctions}, it suffices to show that the derived right Quillen functor, $(\tplus)^!$, is fully faithful. 
  Fix two fibrant objects $(S,S_{hoequiv})$, $(T,T_{hoequiv})$ in $(\sset^+)^{Joy^+}$. 
  We want to prove 
  $$\Map_{\sset^+}((S,S_{hoequiv}),(T,T_{hoequiv})) \to \Map_{\ss^+}((\tplus)^!(S,S_{hoequiv}),(\tplus)^!(T,T_{hoequiv}))$$
  is a Kan equivalence. This will require several computations. 
  \par 
  First, we observe that a $\Delta[n]^\sharp \to (S,S_{hoequiv})$ sends all edges to weak equivalences, and so we have 
  bijection 
  \begin{equation} \label{eq:second marked equation}
   \Hom_{\sset^+}((S,S_{hoequiv}), (T, T_{hoequiv})) \cong \Hom_{\sset}(S,T).
  \end{equation}
  Moreover, we observe that any map of simplicial sets $S \to T$ will automatically send $S_{hoequiv}$ to $T_{hoequiv}$. 
  Giving us a Kan equivalence
  \begin{equation} \label{eq:first marked equation}
   \Hom_{\sset^+}(\Delta[n]^\sharp, (S,S_{hoequiv})) \simeq \Hom_{\sset}(J[n],S).
  \end{equation}  
  Combining these we have
   \begin{align*}
   \Map_{\sset^+}((S,S_{hoequiv}),(T,T_{hoequiv}))_n =
    & \Hom_{\sset^+}((S,S_{hoequiv}) \times \Delta[n]^\sharp, (T, T_{hoequiv}))  & \overset{(\ref{eq:second marked equation})}{\cong} \\
    & \Hom_{\sset^+}((S,S_{hoequiv}) \times (J[n],J[n]_{hoequiv}), (T,T_{hoequiv})) &  \overset{(\ref{eq:first marked equation})}{\simeq} \\
    & \Hom_{\sset}(S \times J[n],T) & = t^!(T^S)_{0n}.
   \end{align*}
  Next, by the argument in the proof of \cref{lemma:t shriek enriched Quillen equiv} we have 
  \begin{equation} \label{eq:third marked equation}
   (\tplus)^!(S,S_hoequiv) = (t^!S,t^!S_{hoequiv}).
  \end{equation}
  Moreover, using the same argument as above, we have a bijection 
  \begin{equation} \label{eq:fourth marked equation}
   \Hom_{\ss^+}((S,S_{hoequiv}), (T, T_{hoequiv})) \cong \Hom_{\ss}(S,T).
  \end{equation}  
  Combining these we have 
  \begin{align*}
   \Map_{\ss^+}((\tplus)^!(S,S_{hoequiv}),(\tplus)^!(T,T_{hoequiv}))_n = 
   & \Hom_{\ss^+}((\tplus)^!(S,S_{hoequiv}) \times \Delta[n],(\tplus)^!(T,T_{hoequiv})) & \overset{(\ref{eq:third marked equation})}{\cong} \\
   & \Hom_{\ss^+}((t^!S,t^!S_{hoequiv}) \times \Delta[n],(t^!T,t^!T_{hoequiv}))  & \overset{(\ref{eq:fourth marked equation})}{\cong}   \\
   & \Hom_{\ss}(t^!S \times \Delta[n],t^!T) & = ((t^!T)^{t^!S})_{0n}.
  \end{align*}
  
  Thus, in order to prove that the derived functor $(\tplus)^!$ is fully faithful it suffices to prove that the map 
  $$t^!(T^S)_0 \to (t^!T^{t^!S})_0$$
  is a Kan equivalence. However, this immediately follows from that fact that the map $t^!(T^S) \to t^!(T)^{t^!(S)}$ is a CSS equivalence
  (\cref{rem:t joyal enriched}). 
  This proves that $((\tplus)_!,(\tplus)^!)$ is a Quillen equivalence. 
  
  We move on to confirm all the properties of the model structure given in the statement of \cref{the:absolute Cartesian model structure on marked simplicial spaces}.
  \begin{enumerate}
  	\item We start by characterizing the fibrant objects. By definition of the localization model structure,
  	an object $(W,A)$ is fibrant if it is fibrant in the unmarked CSS model structure, meaning $W$ is a complete Segal space, 
  	and it lies in the essential image of the right adjoint $(\tplus)^!$, hence $A = W_{hoequiv}$.
  	\item We now move on to prove that cofibrations are monomorphisms. 
  	By \cref{the:localizing adjunctions}, a map is a cofibration if and only its image under $(\tplus)_!$ is a cofibration.
  	However, the cofibrations in the marked Joyal model structure are just the monomorphisms and so the result follows.
  	\item Next we characterize the weak equivalences in the model structure. Notice the marked CSS model structure is simplicial and fibrant objects are of the form $(W,W_{hoequiv})$ and so a map $(Y,A) \to (Z,B)$ is an equivalence if and only if 
  	$$ \Map_{\ss^+}((Z,B),(W,W_{hoequiv})) \to \Map_{\ss^+}((Y,A),(W,W_{hoequiv}))$$
  	is a Kan equivalence for every complete Segal space $W$.
  	\item  Finally, the model structure is given as a localization model structure of the unmarked CSS model structure.
  	Hence, fibrations (weak equivalences) between fibrant objects are given by fibrations (weak equivalences) 
  	in the unmarked CSS model structure. 
  \end{enumerate}
\end{proof}
  We can now use our results about the marked CSS model structure to observe various interesting Quillen equivalences. 

 \begin{theone} \label{the:marked CSS equiv to Cartesian}
  The following diagram 
 \begin{center}
   \begin{tikzcd}[row sep=0.7in, column sep=0.7in]
    \sset^{Joy} \arrow[r, shift left=1.8, "p_1^*", "\bot"'] \arrow[d, shift left=1.8, "(-)^\flat", "\rotatebot"'] & 
    \ss^{CSS} \arrow[d, shift left=1.8, "(-)^\flat", "\rotatebot"'] \arrow[l, shift left=1.8, "i_1^*"]
    \\
    (\sset^+)^{Joy^+} \arrow[r, shift left=1.8, "(\pplus)^*", "\bot"'] \arrow[u, shift left=1.8, "\Forget"] & 
    (\ss^+)^{CSS^+} \arrow[u, shift left=1.8, "\Forget"] \arrow[l, shift left=1.8, "(\iplus)^*"]
   \end{tikzcd}
 \end{center}
 is a commutative diagram of Quillen equivalences. 
 Here the top row has the Joyal and CSS model structure. The bottom row has the 
 marked Joyal and marked CSS model structure. 
 \end{theone}
 
 \begin{proof}
   First we show the diagram commutes.
   It suffices to observe that the left adjoints commute. Moreover, left adjoints commute 
  with colimits and so it suffices to check on generators.  
  However, both sides of the square map $\Delta[n]$ to $F(n)^\flat$.
  
  We now move on to prove all four are Quillen equivalences. 
  By \cref{the:jt ip}, we already know that $(p_1^*,i_1^*)$ is a Quillen equivalence. 
  Similarly, by \cref{the:marked Joyal}, we know that the left hand $((-)^\flat, \Forget)$ is a Quillen equivalence. 
  Hence we have to focus on the other two. 
  
  First we show that $((\pplus)^*, (\iplus)^*)$ is a Quillen adjunction. Evidently $(\pplus)^*$ preserves cofibrations, which are just the monomorphisms. 
  Hence, it suffices to prove that $(\iplus)^*$ preserves fibrant objects and fibrations between fibrant objects. However, fibrant objects are of the form $(W,W_{hoequiv})$, where $W$ is a complete Segal space. By direct computation, $(\iplus)^*(W,W_{hoequiv}) = (i_1^*(W), i_1^*(W)_{hoequiv})$, where $i_1^*(W)$ is a quasi-category (\cref{the:jt ip}). Moreover, fibrations between fibrant objects are just unmarked CSS fibrations (\cref{the:absolute Cartesian model structure on marked simplicial spaces}), which are, again by \cref{the:jt ip}, taken to Joyal fibrations. 
  
  Next, we prove that $((\pplus)^*, (\iplus)^*)$ is a Quillen equivalence. 
  We can extend our adjunction as follows:
  \begin{center}
   \begin{tikzcd}[row sep=0.5in, column sep=0.5in]
    (\sset^+)^{Joy^+} \arrow[r, shift left=1.8, "(\pplus)^*", "\bot"'] & 
    (\ss^+)^{CSS^+} \arrow[r, shift left=1.8, "(\tplus)_!", "\bot"'] \arrow[l, shift left=1.8, "(\iplus)^*"] & 
    (\sset^+)^{Joy^+} \arrow[l, shift left=1.8, "(\tplus)^!"]
   \end{tikzcd}
   .
  \end{center}
  By \cref{the:absolute Cartesian model structure on marked simplicial spaces}, $((\tplus)_!,(\tplus)^!)$ is a Quillen 
  equivalence. Moreover, the composition $(\tplus)_!(\pplus)^*$ is just the identity and so is by default a Quillen equivalence.
  Hence, by $2$-out-of-$3$, $((\pplus)^*,(\iplus)^*)$ is a Quillen equivalence.
  
  Next, we show $((-)^\flat, \Forget)$ is a Quillen adjunction. $(-)^\flat$ preserves cofibrations as they are just the monomorphisms.
  So,we only have to prove that $\Forget$ preserves fibrant objects and fibrations between fibrant objects. 
  Similar to above, fibrant objects are just $(W,W_{hoequiv})$, where $W$ is a complete Segal space and fibrations 
  are just Reedy fibrations, both of which are preserves by $\Forget$. 
  
  Finally, we prove that $((-)^{\flat},\Forget)$ is a Quillen equivalence. 
  However, this follows immediately from the fact that the other three adjunctions are Quillen equivalences and $2$-out-of-$3$.
 \end{proof} 
 
 Combining \cref{the:absolute Cartesian model structure on marked simplicial spaces} and \cref{the:marked CSS equiv to Cartesian} 
 with \cref{prop:quillen equiv over cat} gives us following useful corollary.
 
 \begin{corone} \label{cor:marked quillen equiv over}
  Let $(X,A)$ be a marked simplicial space and $(S,B)$ be a marked simplicial set. We have following Quillen equivalences
  \begin{center}
   \adjun{(\ss^+)^{CSS^+}_{/X}}{(\sset^+)^{Joy^+}_{/(\tplus)_!(X)}}{(\tplus)_!}{u^*(\tplus)^!}
   \adjun{(\sset^+)^{Joy^+}_{/S}}{(\ss^+)^{CSS^+}_{/ (\pplus)^*S}}{(\pplus)^*}{u^*(\iplus)^*}.
  \end{center}
  Moreover, the adjunction $((\tplus)_!,u^*(\tplus)^!)$ is a simplicial Quillen equivalence. 
 \end{corone}
 Taking $X$ to be the final object, the result above implies that we have constructed another model structure for $(\infty,1)$-categories on the category of marked simplicial spaces.
 
 \begin{remone} \label{rem:two model structures}
  We can summarize the results of this subsection with the following diagram of Quillen equivalences 
  \begin{center}
   \begin{tikzcd}[row sep=0.5in, column sep=0.5in]
     & (\ss^+)^{un^+CSS} \arrow[dl, shift left=1.8, "\Forget"] \arrow[dr, shift left=1.8, "(\tplus)_!", "\rotateothertiltbot"' xshift=0.11in] 
     \arrow[dd, shift left=1.8, "\mathrm{id}", "\rotatebot"'] & \\
     \ss^{CSS} \arrow[ur, shift left=1.8, "(-)^\flat", "\rotatetiltbot"' xshift=-0.11in] 
     \arrow[dr, shift left=1.8, "(-)^\flat", "\rotateothertiltbot"' xshift=0.11in] & & 
     (\sset^+)^{Joy^+} \arrow[ul, shift left=1.8, "(\tplus)^!"] \arrow[dl, shift left=1.8, "(\tplus)^!"] \\
     & (\ss^+)^{CSS^+} \arrow[ul, shift left=1.8, "\Forget"] 
     \arrow[ur, shift left=1.8, "(\tplus)_!", "\rotatetiltbot"' xshift=-0.11in] \arrow[uu, shift left=1.8, "\mathrm{id}"] & 
   \end{tikzcd}
  \end{center}
  The CSS model structure on the left hand side, $\ss^{CSS}$, and the marked Joyal model structure on the right hand side,  $(\sset^+)^{Joy^+}$, 
  already existed and we build a bridge from one to the other by defining the marked CSS model structure, $(\ss^+)^{CSS^+}$.
  
  So, what is the role of the unmarked CSS model structure,$(\ss^+)^{un^+CSS}$? 
  If we wanted to directly construct a model structure on $\ss^+$ using the adjunction $((\tplus)_!, (\tplus)^!)$, we would 
  have to {\it left induce} our model structure from the marked Joyal model structure. 
  However, left-induced model structures are quite difficult to construct \cite{hkrs2017leftinduced}. 
  
  Hence, we used the unmarked CSS model structure, $(\ss^+)^{un^+CSS}$, as an intermediate step 
  to avoid such difficulties. We can easily right induce this
  model structure from the CSS model structure, $\ss^{CSS}$ (as we did in \cref{prop:transferred model structure on marked simplicial spaces}). 
  Now, as soon as we have {\it any} appropriate model structure on $\ss^+$, we can then 
  use the theory of Bousfield localizations to get the desired marked CSS model structure, $(\ss^+)^{CSS^+}$.
  Thus, we can think of the unmarked CSS model structure as an intermediate step that allows us to easily define 
  the marked CSS model structure. 
 \end{remone}
 
 We will use these adjunctions in the next subsection to define and study the Cartesian model structure of 
 marked simplicial spaces.

 \subsection{Marked Simplicial Spaces vs. Marked Simplicial Sets}
  In this section we define Cartesian fibrations of simplicial spaces and prove that we can define a 
  model structure on marked simplicial spaces over a simplicial space such that the fibrant objects are precisely the Cartesian fibrations. 
  We also show that this model structure is Quillen equivalent to the Cartesian model structure on marked simplicial sets via both adjunctions 
  $((\tplus)_!, (\tplus)^!)$, $((\pplus)^*, (\iplus)^*)$.
  
 First, we review the definition of Cartesian fibrations and Cartesian model structure of simplicial sets, as discussed in \cite{lurie2009htt}.
 
 \begin{defone} \label{def:cartesian mor qcat}
 \cite[Definition 2.4.1.1.]{lurie2009htt}
  Let $p: T \to S$ be a map of simplicial sets. An edge $f:x \to y$ in $T$ is $p$-Cartesian if 
  $$ T_{/f} \twoheadrightarrow S_{/p(f)} \times_{S_{/p(y)}} T_{/y}$$
  is a trivial Kan fibration. Here, $T_{/y}$ is defined in \cite[Proposition 1.2.9.2]{lurie2009htt}.
 \end{defone}

 \begin{defone} \label{def:cart fib sset}
  An inner fibration $p: T \to S$ is a {\it Cartesian fibration}, if for every edge $f: x \to y$ in $S$ and 
  lift $\tilde{y}$, there exist a $p$-Cartesian lift $\tilde{f}: \tilde{x} \to \tilde{y}$.
 \end{defone}

 Cartesian fibrations are fibrant objects in a model structure on marked simplicial sets. 
 
 \begin{theone} \label{the:cart model structure marked simp set}
  \cite[Proposition 3.1.3.7, Proposition 3.1.4.1, Corollary 3.1.4.4, Proposition 3.1.5.3]{lurie2009htt}
  Let $S$ be a simplicial set. There exists a unique left proper, combinatorial, simplicial model structure on 
  $\sset^+_{/S^\#}$ which may be described as follows:
  \begin{itemize}
   \item[C] A map is a cofibration if the map of underlying simplicial sets is a monomorphism.
   \item[F] Fibrant objects are of the form  $(T,T^{Cart}) \to S^\#$, where $T \to S$ is a Cartesian fibration and $T^{Cart}$ the 
   set of Cartesian edges.
   \item[W] A morphism $g$ is an equivalence if for every Cartesian fibration $T \to S$ the induced map 
   $\Map_{/S}(g,T)$ is a Kan equivalence.
  \end{itemize}
  The simplicial enrichment is given by 
  $$\Map_{/S^\#}(T,U)_n = \Hom_{/S^\#}(T \times \Delta[n]^\#,U).$$
  Finally, for every simplicial set $S$ the adjunction 
  \begin{center}
   \adjun{(\sset_{/S^\#})^{Joy^+}}{(\sset^+_{/S^{\#}})^{Cart}}{\id}{\id}
  \end{center}
  is a Quillen adjunction between the Joyal model structure and the Cartesian model structure. 
 \end{theone}
 
  \begin{notone}
  In \cite{lurie2009htt} a map of the form  $(T,T^{Cart}) \to S^\#$ where $T \to S$ is a Cartesian fibration 
  is denoted by $T^\natural$ over $S$, however, we will not use this notation.
 \end{notone}
 
 \begin{exone}
  If $S=\Delta[0]$ then the Cartesian model structure simply recovers the marked Joyal model structure introduced in 
  \cref{the:marked Joyal}.
 \end{exone}

 Our goal is to define Cartesian fibrations of marked simplicial spaces and prove it comes with a model structure
 that is equivalent to the Cartesian model structure on marked simplicial sets. 
  
 \begin{defone} \label{def:cart fib simplicial spaces}
  A CSS fibration of simplicial spaces $Y \to X$ is a Cartesian fibration if
  $i_1^*Y \to i_1^*X$ is a Cartesian fibration of simplicial sets.
 \end{defone}
 
 We now want to give an internal description of Cartesian fibrations of simplicial spaces, that does not rely on simplicial sets. 
 This requires some category theory of simplicial spaces, as studied in \cite{rasekh2017left}.
 
 Let $X$ be a complete Segal space and $f$ a morphism. Then the the {\it complete Segal space of cones} $X_{/f}$ \cite[Definition 5.21]{rasekh2017left} is defined as
 $$X_{/f} = (X^{F(1)})^{F(1)} \underset{X^{F(1)} \times X^{F(1)}}{\times} X \times F(0).$$ 
 
 Generalizing from there, if $X$ is an arbitrary simplicial space, then we can choose a complete Segal space fibrant replacement $X \to \hat{X}$ 
 and define $X_{/f} = X \times_{\hat{X}} \hat{X}_{/f}$, using the fact that cocones are invariant under equivalences of complete Segal spaces \cite[Theorem 5.1]{rasekh2017left}. Note that, up to Reedy equivalence of simplicial spaces, this definition is independent 
 of the choice of $\hat{X}$.
 
 \begin{defone} \label{def:pcart morphism}
  Let $p:V \to W$ be a CSS fibration of simplicial spaces. A morphism $f: x \to y$ in $V$ is called $p$-Cartesian if the
  commutative square 
  \begin{center}
   \begin{tikzcd}
    V_{/f} \arrow[r] \arrow[d] & V_{/y} \arrow[d] \\
    W_{/p(f)} \arrow[r] & W_{/p(y)}
   \end{tikzcd}
  \end{center}
  is a homotopy pullback square of simplicial spaces. 
 \end{defone}
  
  We can now use this to give an internal characterization of Cartesian fibrations.
  
 \begin{lemone} \label{lemma:cart fib def agree}
  Let $p: V \to W$ be CSS fibration. Then $p$ is a Cartesian fibration if and only if 
  for every morphism $f: x \to y$ in $W$ and lift $\hat{y}$ in $V$ of $y$, there exists a $p$-Cartesian morphism lift 
  $\hat{f}$ of $f$ (meaning $p(\hat{f}) = f$).
 \end{lemone}
 
 \begin{proof}
  Before we start the proof, we recall that $i_1^*(W)_0 = W_{10}$, which means $W$ and $i_1^*W$ have the same set of edges.
  
  First, by \cref{the:jt ip}, if $p: V \to W$ is a CSS fibration then $i_1^*(p): i_1^*(V) \to i_1^*(W)$ is a fibration 
  in the Joyal model structure. Thus, by \cref{def:cart fib sset}, $i_1^*(p): i_1^*(V) \to i_1^*(W)$ is a Cartesian fibration of simplicial sets 
  if and only if every edge $f: x \to y$ in $i_1^*(W)$ with given lift $\hat{y}$ has a $i_1^*(p)$-Cartesian lift. 
  
  Hence the statement follows immediately from proving that an edge $f$ in $V$ is $p$-Cartesian if and only if 
  $f$ in $i_1^*(V)$ is $i_1^*(p)$-Cartesian.
  
  However, this follows from the fact that $i_1^*$ is a right Quillen functor of a Quillen equivalence and so 
    \begin{center}
   \begin{tikzcd}
    V_{/f} \arrow[r] \arrow[d] & V_{/y} \arrow[d] & & i_1^*(V)_{/f} \arrow[r] \arrow[d] & i_1^*(V)_{/y} \arrow[d]\\
    W_{/p(f)} \arrow[r] & W_{/p(y)} & & i_1^*(W)_{/p(f)} \arrow[r] & i_1^*(W)_{/p(y)}
   \end{tikzcd}
  \end{center}
  the right square is a homotopy pullback square if and only if the left hand square is a homotopy pullback square.
 \end{proof}
 
 We now want to move our study of Cartesian fibrations into the marked world and so need an appropriate definition. 
 
 \begin{defone} \label{def:cart fib marked simp spaces}
  Let $X$ be a simplicial space. We say a map of marked simplicial space $(Y,Y^{Cart}) \to X^\sharp$ is a 
  {\it Cartesian fibration} if $Y \to X$ is a Cartesian fibration and $Y^{Cart}$ is the sub-space 
  of Cartesian edges.
 \end{defone}
 
 We are now ready to show that 
 Cartesian fibrations (\cref{def:cart fib marked simp spaces}) are fibrant objects in a model structure.
  
  \begin{theone} \label{the:cart model structure marked simp spaces}
   Let $X$ be a simplicial space. There is a unique combinatorial, left proper, simplicial model structure on $(\ss^+)_{/X^\#}$, 
   called the {\it Cartesian model structure} and denoted $(\ss_{/X^\#})^{Cart}$, such that
   \begin{enumerate}
    \item Cofibrations are monomorphisms over $X^\#$.
    \item Fibrant objects are Cartesian fibrations $(Y,Y^{Cart}) \to X^\#$.
    \item A map is a weak equivalence if for every Cartesian fibration $(W,D) \to X^\sharp$ the induced map 
    $$\Map_{/X^\sharp}((Z,C), (W,D)) \to \Map_{/X^\sharp}((Y,B), (W,D)) $$
    is a Kan equivalence.
    \item A map between Cartesian fibrations $(W,W^{Cart}) \to (V,V^{Cart}))$ over $X^\sharp$ is 
    a weak equivalence (fibration) if and only if it is a weak equivalence (fibration) in the marked CSS model structure.
    \item The adjunction
    \begin{center}
     \adjun{((\ss^+)_{/X^\#})^{Cart}}{((\sset^+)_{/(\tplus)_!X^\#})^{Cart}}{(\tplus)_!}{u^*(\tplus)^!}
    \end{center}
    is a simplicial Quillen equivalence.
   \end{enumerate}
  \end{theone}
  
  \begin{proof}
   Combining \cref{cor:marked quillen equiv over} and \cref{the:cart model structure marked simp set}, we have simplicial Quillen adjunctions 
   \begin{equation} \label{eq:adjunctions}
    \begin{tikzcd}
     ((\ss^+)_{/X^\#})^{CSS^+} \arrow[r, shift left=1.8, "(\tplus)_!", "\bot"'] & 
     ((\sset^+)_{/(\tplus)_!X^\#})^{Joy^+} \arrow[r, shift left=1.8, "\id", "\bot"'] \arrow[l, shift left=1.8, "u^*(\tplus)^!"] &
     ((\sset^+)_{/(\tplus)_!X^\#})^{Cart} \arrow[l, shift left=1.8, "\id"]
    \end{tikzcd}
   \end{equation}
   between combinatorial simplicial model structures, where the right adjoint is fully faithful 
   (by \cref{cor:marked quillen equiv over}). Hence, by \cref{the:localizing adjunctions}, 
   we can define a new model structure on $\ss^+_{/X^\#}$, which we call the Cartesian model structure and that makes the 
   adjunction $((\tplus)_!,u^*(\tplus)^!)$ into a simplicial Quillen equivalence.
   
  Next, combining the facts that $(\tplus)_!$ reflects cofibrations (\cref{the:localizing adjunctions}) and that 
  cofibrations in $((\sset^+)_{/t_!X^\#})^{Cart}$ are monomorphisms implies that cofibrations in 
  $((\ss^+)_{/X^\#})^{Cart}$ are also monomorphisms.
  
  Next, we characterize the fibrant objects, which are precisely the fibrant objects in $(\ss^+_{/X})^{CSS^+}$ that are in the essential 
  image of $u^*(\tplus)^!$. By \cref{the:marked CSS equiv to Cartesian}, 
  if a map $(Y,B) \to X^\sharp$ is a marked CSS fibration, then $Y \to X$ is a CSS fibration of simplicial spaces. 
  On the other hand, $(Y,B)$ is in the image if $(Y,B) = (\tplus)^*(C,C^{Cart}) = (t^!C, (t^!C)^{Cart})$, which implies that 
  $B = Y^{Cart}$. Hence, the fibrant objects are precisely the Cartesian fibrations over $X^\#$.
    
  The characterization of weak equivalences follows from the fact that the model structure is simplicial and 
  that the fibrant objects are given by Cartesian fibrations of marked simplicial spaces. 
  
  Finally, the model structure is given as a localization model structure of the marked CSS model structure.
  Hence, fibrations (weak equivalences) between fibrant objects are given by fibrations (weak equivalences) 
  in the marked CSS model structure. 
\end{proof}

  We constructed the Cartesian model structure on $(\ss^+)_{/X}$ as a localization model structure. 
  Hence, we get following result.
  
  \begin{corone} \label{cor:cart local of css}
   Let $X$ be a simplicial space. Then the adjunction 
   \begin{center}
    \adjun{((\ss^+)_{/X})^{CSS^+}}{((\ss^+)_{/X})^{Cart}}{\id}{\id}
   \end{center}
   is a Quillen adjunction.  
  \end{corone}

  We end this section by giving a second Quillen equivalence between the Cartesian model structures. 
  
 \begin{theone} \label{the:ip cart equiv}
  Let $S$ be a simplicial set. The adjunction 
  \begin{center}
   \adjun{((\sset^+)_{/S})^{Cart}}{((\ss^+)_{/(\pplus)^*S})^{Cart}}{(\pplus)^*}{u^*(\iplus)^*}
  \end{center}
  is a Quillen equivalence between Cartesian model structures.
 \end{theone}
 
 \begin{proof}
  In order to prove it is a Quillen adjunction it suffices to show the left adjoint preserves cofibrations and 
  the right adjoint preserves fibrant objects and trivial fibrations between fibrant objects. 
  Evidently, $(\pplus)^*$ preserves cofibrations as they are just monomorphisms. 
  
  Moreover, $(\iplus)^*$ preserves trivial fibrations between fibrant objects.
  Indeed, by \cref{the:cart model structure marked simp spaces}, weak equivalences between fibrant objects are 
  just marked CSS equivalences, which, by \cref{the:marked CSS equiv to Cartesian}, are mapped to marked Joyal equivalences, 
  which, by \cref{the:cart model structure marked simp set}, are Cartesian equivalences. 
  
  We are left with proving that $(\iplus)^*$ preserves fibrant object. However, by \cref{the:cart model structure marked simp spaces},
  fibrant objects are of the form  $(Y,Y^{Cart}) \to X^\sharp$, where $ Y \to X$ is a Cartesian fibration. 
  By direct computation, $(\iplus)^*((Y,Y^{Cart}) \to X^\sharp) = (i_1^*Y, i_1^*Y^{Cart}) \to i_1^*X^\sharp$, which 
  is fibrant in the Cartesian model structure in $(\sset^+_{/i_1^*X^\sharp})^{Cart}$ (\cref{the:cart model structure marked simp set}).
  
  We move on to prove that the adjunction is a Quillen equivalence.
  We can extend the adjunction above to the following diagram
  \begin{center}
   \begin{tikzcd}[row sep=0.5in, column sep=0.5in]
    ((\sset^+)_{/S^\#})^{Cart} \arrow[r, shift left=1.8, "(\pplus)^*", "\bot"'] & 
    ((\ss^+)_{/(\pplus)^*S^\#})^{Cart} \arrow[r, shift left=1.8, "(\tplus)_!", "\bot"'] \arrow[l, shift left=1.8, "u^*(\iplus)^*"]&
    ((\sset^+)_{/S^\#})^{Cart} \arrow[l, shift left=1.8, "u^*(\tplus)^!"]
   \end{tikzcd}.
  \end{center}
  The second adjunction is a Quillen equivalence by the previous theorem. The composition is the identity and 
  so also a Quillen equivalence. Hence the left hand Quillen adjunction is a Quillen equivalence, by $2$-out-of-$3$.
 \end{proof}

\section{Equivalence of Cartesian Model Structures: Model Dependent and Independent} \label{sec:equivalence of Cartesian model structures}
 Until this point we have studied many different model structures, that we all denoted by {\it Cartesian model structure}: 
 for marked simplicial sets, marked simplicial spaces, bisimplicial sets and bisimplicial spaces.
 We are finally in a position where we can show that all these model structures are in fact Quillen equivalent 
 via a zig-zag of Quillen equivalence and hence we are justified in giving all of them the same name.
 
 Finally, in the last subsection, we show how this immediately gives us an equivalence of underlying quasi-categories, 
 where we can replace a zig-zag of equivalences with direct adjoint equivalences. 
 
\subsection{Quillen Equivalence of Cartesian Model Structures} \label{subsec:quillen Equivalence of Cartesian Model Structures}
 
 In this section we want to prove that we have following diagram of Quillen equivalences (the direction of the arrows are the direction of 
 the right adjoints).
 Here an arrow marks a right Quillen functor of a Quillen equivalence, $X$ is a simplicial space and $S$ a simplicial set. 
 \begin{center}
  \begin{tikzcd}[row sep=0.5in, column sep=0.8in]
   \text{Marked Simplicial} & \text{Functorial} & \text{Bisimplicial} \\[-0.55in] \hline
   \\[-0.3in]
   (\sset^+_{/t_!X})^{Cart} \arrow[d, "\arrowref{qe:t}{((\tplus)_! \comma u^*(\tplus)^!)}" description] & 
   \Fun(\mathfrak{C}[t_!X], (\sset^+)^{Joy^+})^{proj} 
   \arrow[l, "\arrowref{qe:stunplus}{(\St_{t_!X}^+ \comma \Un_{t_!X}^+)}" description] 
   \arrow[d, "\arrowref{qe:betastar}{(((\tplus)_!)_* \comma ((\tplus)^!)_*)}" description]  & \\
   (\ss^+_{/X})^{Cart}  & \Fun(\mathfrak{C}[t_!X], (\ss^+)^{CSS^+})^{proj}
   \arrow[d, "\arrowref{qe:tstar}{(((-)^\flat)_* \comma (\Forget)_*)}" description]& 
   (\sss_{/X})^{Cart}  \\
   &  \Fun(\mathfrak{C}[t_!X], \ss^{CSS})^{proj}  
   \arrow[r, "\arrowref{qe:stun}{(\sSt_{t_!X} \comma \sUn_{t_!X})}" description]  & 
   (\ssset_{/t_!X})^{Cart} \arrow[u, "\arrowref{qe:st}{(st_! \comma su^*t^!)}" description] \\
   \\[-0.4in]
   (\sset^+_{/S})^{Cart} & 
   \Fun(\mathfrak{C}[S], (\sset^+)^{Joy^+})^{proj} \arrow[l, "\arrowref{qe:stunplus}{(\St_{S}^+ \comma \Un_{S}^+)}" description] 
   \arrow[d, "\arrowref{qe:betastar}{(((\tplus)_!)_* \comma ((\tplus)^!)_*)}" description]  \\
   (\ss^+_{/p_1^*S})^{Cart} \arrow[u, "\arrowref{qe:pi}{((\pplus)^* \comma u^*(\iplus)^*)}" description] & 
   \Fun(\mathfrak{C}[S], (\ss^+)^{CSS^+})^{proj}  
   \arrow[d, "\arrowref{qe:tstar}{(((-)^\flat)_* \comma (\Forget)_*)}" description] & 
   (\sss_{/p_1^*S})^{Cart} \arrow[d, "\arrowref{qe:spi}{(sp_1^* \comma su^*i_1^*)}" description] \\
   &  \Fun(\mathfrak{C}[S], \ss^{CSS})^{proj}  
   \arrow[r, "\arrowref{qe:stun}{(\sSt_{S} \comma \sUn_{S})}" description] & 
   (\ssset_{/S})^{Cart}  
  \end{tikzcd}
 \end{center}
  
  \setword{$\bullet {\bf ((\tplus)_!,u^*(\tplus)^!):}$}{qe:t}
  This is proven in \cref{the:cart model structure marked simp spaces}.
  
  \medskip

  \setword{${\bf \bullet ((\pplus)^*,u^*(\iplus)^*):}$}{qe:pi}
  This is proven in \cref{the:ip cart equiv}.
  \medskip 
  
   \setword{${\bf \bullet (st_!,su^*t^!):}$}{qe:st}
  This is proven in \cref{the:cart model structure bisimp set}.
  
  \medskip 
  
  \setword{${\bf \bullet (sp_1^*,su^*i_1^*):}$}{qe:spi}
  This is proven in \cref{the:cart model structure bisimp set}.
  
  \medskip

  \setword{${\bf \bullet (((\tplus)_!)_* , ((\tplus)^!)_*):}$}{qe:betastar}
  This follows from applying \cref{lemma:projective model invariant} to the Quillen equivalence $(((\tplus)_!)_*,((\tplus)_!)_*)$ 
  (\cref{the:absolute Cartesian model structure on marked simplicial spaces}).
  
  \medskip 
    
  \setword{$\bullet {\bf (((-)^\flat)_* , (\Forget)_*):}$}{qe:tstar}
  This follows from applying \cref{lemma:projective model invariant} to the Quillen equivalence $(((-)^\flat)_*,(\Forget)_*)$ 
  (\cref{the:marked CSS equiv to Cartesian}).
  
  \medskip
  
  \setword{${\bf \bullet (\sSt_S,\sUn_S):}$}{qe:stun}
  This follows from applying \cref{the:equiv of CSO} to the Quillen equivalence $(\St_S,\Un_S)$ \cite[Theorem 2.2.1.2]{lurie2009htt}.
  
  \medskip 
  
  \setword{${\bf \bullet (\St_S^+,\Un_S^+):}$}{qe:stunplus}
  This is the statement of \cite[Theorem 3.2.0.1]{lurie2009htt}.
  
  \medskip 
  
\subsection{Equivalences of Quasi-Categories} \label{subsec:equivalences of Quasi-Categories}
 The work here focused on model categorical formalism, but in this subsection we want to give short analysis of the 
 quasi-categorical implications of this result. The model structures from 
 \cref{subsec:quillen Equivalence of Cartesian Model Structures} are all simplicial. So, we can apply the simplicial nerve 
 $N$ (\cite[1.1.5]{lurie2009htt}) and the the fact that Quillen equivalences of simplicial model categories gives us adjoint equivalences of 
 quasi-categories \cite[Proposition 5.2.4.6]{lurie2009htt}, to get a diagram of equivalences of quasi-categories 
 (here we only focus on the case over a simplicial space $X$):
 
  \begin{center}
  \begin{tikzcd}[row sep=0.5in, column sep=0.8in]
   N\left((\sset^+_{/t_!X})^{Cart}\right) \arrow[d, "((\tplus)_! \comma u^*(\tplus)^!)" description, dash] & 
   N\left(\Fun(\mathfrak{C}[t_!X], (\sset^+)^{Cart})^{proj}\right) 
   \arrow[l, "(\St_{t_!X}^+ \comma \Un_{t_!X}^+)" description, dash] 
   \arrow[d, "(((\tplus)_!)_* \comma ((\tplus)^!)_*)" description, dash]  & \\
   N\left((\ss^+_{/X})^{Cart}\right)  & N\left(\Fun(\mathfrak{C}[t_!X], (\ss^+)^{CSS^+})^{proj}\right)
   \arrow[d, "(((-)^\flat)_* \comma (\Forget)_*)" description, dash]& 
   N\left((\sss_{/X})^{Cart}\right)  \\
   &  N\left(\Fun(\mathfrak{C}[t_!X], \ss^{CSS})^{proj}\right)  
   \arrow[r, "(\sSt_{t_!X} \comma \sUn_{t_!X})" description, dash]  & 
   N\left((\ssset_{/t_!X})^{Cart}\right) \arrow[u, "(st_! \comma su^*t^!)" description, dash]
  \end{tikzcd}
 \end{center}
 By \cite[2.1.12]{riehlverity2018elements}, whenever we have an adjoint equivalence of quasi-categories, both functors 
 are right and left adjoints. Thus we can ignore the direction of the arrows. This is clearly not true for Quillen equivalences 
 of model categories, as functors (even equivalences) are usually not right and left Quillen.

 \bibliographystyle{alpha}
 \bibliography{main}

\begin{thebibliography}{HKRS17}

\bibitem[AF20]{ayalafrancis2020fibrations}
David Ayala and John Francis.
\newblock Fibrations of {$\infty$}-categories.
\newblock {\em High. Struct.}, 4(1):168--265, 2020.

\bibitem[Ara14]{ara2014highersegal}
Dimitri Ara.
\newblock Higher quasi-categories vs higher {R}ezk spaces.
\newblock {\em J. K-Theory}, 14(3):701--749, 2014.

\bibitem[Bar05]{barwick2005nfoldsegalspaces}
Clark Barwick.
\newblock {\em (infinity, n)-{C}at as a closed model category}.
\newblock ProQuest LLC, Ann Arbor, MI, 2005.
\newblock Thesis (Ph.D.)--University of Pennsylvania.

\bibitem[BdB18]{debrito2018leftfibration}
Pedro Boavida~de Brito.
\newblock Segal objects and the {G}rothendieck construction.
\newblock In {\em An alpine bouquet of algebraic topology}, volume 708 of {\em
  Contemp. Math.}, pages 19--44. Amer. Math. Soc., Providence, RI, 2018.

\bibitem[Ber10]{bergner2010survey}
Julia~E. Bergner.
\newblock A survey of {$(\infty,1)$}-categories.
\newblock In {\em Towards higher categories}, volume 152 of {\em IMA Vol. Math.
  Appl.}, pages 69--83. Springer, New York, 2010.

\bibitem[Ber20]{bergner2020surveyn}
Julia~E Bergner.
\newblock A survey of models for $(\infty, n)$-categories.
\newblock {\em Handbook of Homotopy Theory, edited by Haynes Miller, Chapman \&
  Hall/CRC}, pages 263--295, 2020.

\bibitem[BR13]{bergnerrezk2013comparisoni}
Julia~E. Bergner and Charles Rezk.
\newblock Comparison of models for {$(\infty,n)$}-categories, {I}.
\newblock {\em Geom. Topol.}, 17(4):2163--2202, 2013.

\bibitem[BR20]{bergnerrezk2020comparisonii}
Julia~E. Bergner and Charles Rezk.
\newblock Comparison of models for {$(\infty, n)$}-categories, {II}.
\newblock {\em J. Topol.}, 13(4):1554--1581, 2020.

\bibitem[BV73]{boardmanvogt1973qcats}
J.~M. Boardman and R.~M. Vogt.
\newblock {\em Homotopy invariant algebraic structures on topological spaces}.
\newblock Lecture Notes in Mathematics, Vol. 347. Springer-Verlag, Berlin-New
  York, 1973.

\bibitem[CKM20]{campionkapulkinmaehara2020comical}
Timothy Campion, Krzysztof Kapulkin, and Yuki Maehara.
\newblock Comical sets: A cubical model for $(\infty,n)$-categories.
\newblock {\em arXiv preprint}, 2020.
\newblock \href{https://arxiv.org/abs/2005.07603v2}{arXiv:2005.07603v2}.

\bibitem[Con07]{conrad2007moduliellipticcurves}
Brian Conrad.
\newblock Arithmetic moduli of generalized elliptic curves.
\newblock {\em J. Inst. Math. Jussieu}, 6(2):209--278, 2007.

\bibitem[CS19]{calaquescheimbauer2019cobordism}
Damien Calaque and Claudia Scheimbauer.
\newblock A note on the {$(\infty,n)$}-category of cobordisms.
\newblock {\em Algebr. Geom. Topol.}, 19(2):533--655, 2019.

\bibitem[GHL20]{gagnaharpazlanari2020inftytwolimits}
Andrea Gagna, Yonatan Harpaz, and Edoardo Lanari.
\newblock Fibrations and lax limits of $(\infty,2)$-categories.
\newblock {\em arXiv preprint}, 2020.
\newblock \href{https://arxiv.org/abs/2012.04537}{arXiv:2012.04537}.

\bibitem[GJ99]{goerssjardine1999simplicialhomotopytheory}
Paul~G. Goerss and John~F. Jardine.
\newblock {\em Simplicial homotopy theory}, volume 174 of {\em Progress in
  Mathematics}.
\newblock Birkh\"auser Verlag, Basel, 1999.

\bibitem[GR17a]{gaitsgoryrozenblyum2017dagI}
Dennis Gaitsgory and Nick Rozenblyum.
\newblock {\em A study in derived algebraic geometry. {V}ol. {I}.
  {C}orrespondences and duality}, volume 221 of {\em Mathematical Surveys and
  Monographs}.
\newblock American Mathematical Society, Providence, RI, 2017.

\bibitem[GR17b]{gaitsgoryrozenblyum2017dagII}
Dennis Gaitsgory and Nick Rozenblyum.
\newblock {\em A study in derived algebraic geometry. {V}ol. {II}.
  {D}eformations, {L}ie theory and formal geometry}, volume 221 of {\em
  Mathematical Surveys and Monographs}.
\newblock American Mathematical Society, Providence, RI, 2017.

\bibitem[gro03]{grothendieck2003etalegroup}
{\em Rev\^{e}tements \'{e}tales et groupe fondamental ({SGA} 1)}, volume~3 of
  {\em Documents Math\'{e}matiques (Paris) [Mathematical Documents (Paris)]}.
\newblock Soci\'{e}t\'{e} Math\'{e}matique de France, Paris, 2003.
\newblock S\'{e}minaire de g\'{e}om\'{e}trie alg\'{e}brique du Bois Marie
  1960--61. [Algebraic Geometry Seminar of Bois Marie 1960-61], Directed by A.
  {G}rothendieck, With two papers by M. Raynaud, Updated and annotated reprint
  of the 1971 original [Lecture Notes in Math., 224, Springer, Berlin;
  MR0354651 (50 \#7129)].

\bibitem[Gro10]{groth2010inftycategories}
Moritz Groth.
\newblock A short course on $\infty$-categories.
\newblock {\em arXiv preprint}, 2010.
\newblock \href{https://arxiv.org/abs/1007.2925v2}{arXiv:1007.2925v2}.

\bibitem[GS07]{goerssschemmerhorn2007modelcats}
Paul Goerss and Kristen Schemmerhorn.
\newblock Model categories and simplicial methods.
\newblock In {\em Interactions between homotopy theory and algebra}, volume 436
  of {\em Contemp. Math.}, pages 3--49. Amer. Math. Soc., Providence, RI, 2007.

\bibitem[Hir03]{hirschhorn2003modelcategories}
Philip~S. Hirschhorn.
\newblock {\em Model categories and their localizations}, volume~99 of {\em
  Mathematical Surveys and Monographs}.
\newblock American Mathematical Society, Providence, RI, 2003.

\bibitem[HKRS17]{hkrs2017leftinduced}
Kathryn Hess, Magdalena K\c{e}dziorek, Emily Riehl, and Brooke Shipley.
\newblock A necessary and sufficient condition for induced model structures.
\newblock {\em J. Topol.}, 10(2):324--369, 2017.

\bibitem[HM15]{heutsmoerdijk2015leftfibrationi}
Gijs Heuts and Ieke Moerdijk.
\newblock Left fibrations and homotopy colimits.
\newblock {\em Math. Z.}, 279(3-4):723--744, 2015.

\bibitem[HM16]{heutsmoerdijk2016leftfibrationii}
Gijs Heuts and Ieke Moerdijk.
\newblock Left fibrations and homotopy colimits ii.
\newblock {\em arXiv preprint}, 2016.
\newblock \href{https://arxiv.org/abs/1602.01274v1}{arXiv:1602.01274v1}.

\bibitem[Joh02]{johnstone2002elephanti}
Peter~T. Johnstone.
\newblock {\em Sketches of an elephant: a topos theory compendium. {V}ol. 1},
  volume~43 of {\em Oxford Logic Guides}.
\newblock The Clarendon Press, Oxford University Press, New York, 2002.

\bibitem[Joy08a]{joyal2008notes}
Andr{\'e} Joyal.
\newblock Notes on quasi-categories.
\newblock {\em preprint}, 2008.
\newblock \href{https://www.math.uchicago.edu/~may/IMA/Joyal.pdf}{Unpublished
  notes} (accessed 08.02.2021).

\bibitem[Joy08b]{joyal2008theory}
Andr{\'e} Joyal.
\newblock The theory of quasi-categories and its applications.
\newblock 2008.
\newblock
  \href{https://mat.uab.cat/~kock/crm/hocat/advanced-course/Quadern45-2.pdf}{Unpublished
  notes} (accessed 08.02.2021).

\bibitem[JT07]{joyaltierney2007qcatvssegal}
Andr\'{e} Joyal and Myles Tierney.
\newblock Quasi-categories vs {S}egal spaces.
\newblock In {\em Categories in algebra, geometry and mathematical physics},
  volume 431 of {\em Contemp. Math.}, pages 277--326. Amer. Math. Soc.,
  Providence, RI, 2007.

\bibitem[Lur04]{lurie2004dag}
Jacob Lurie.
\newblock {\em Derived algebraic geometry}.
\newblock ProQuest LLC, Ann Arbor, MI, 2004.
\newblock Thesis (Ph.D.)--Massachusetts Institute of Technology.

\bibitem[Lur09a]{lurie2009htt}
Jacob Lurie.
\newblock {\em Higher topos theory}, volume 170 of {\em Annals of Mathematics
  Studies}.
\newblock Princeton University Press, Princeton, NJ, 2009.

\bibitem[Lur09b]{lurie2009goodwillie}
Jacob Lurie.
\newblock (infinity, 2)-categories and the goodwillie calculus i.
\newblock {\em arXiv preprint}, 2009.
\newblock \href{https://arxiv.org/abs/0905.0462v2}{arXiv:0905.0462v2}.

\bibitem[Lur09c]{lurie2009cobordism}
Jacob Lurie.
\newblock On the classification of topological field theories.
\newblock In {\em Current developments in mathematics, 2008}, pages 129--280.
  Int. Press, Somerville, MA, 2009.

\bibitem[Lur17]{lurie2017ha}
Jacob Lurie.
\newblock Higher algebra.
\newblock \href{http://www.math.ias.edu/~lurie/papers/HA.pdf}{Unpublished
  notes} (accessed 08.02.2021), September 2017.

\bibitem[nla20]{nlab2020marked}
model structure for cartesian fibrations.
\newblock nLab, 2020.
\newblock
  URL:\href{https://ncatlab.org/nlab/show/model+structure+for+Cartesian+fibrations}{nlab}
  (version: 08.02.2021).

\bibitem[Qui67]{quillen1967modelcats}
Daniel~G. Quillen.
\newblock {\em Homotopical algebra}.
\newblock Lecture Notes in Mathematics, No. 43. Springer-Verlag, Berlin-New
  York, 1967.

\bibitem[Ras17a]{rasekh2017cartesianrep}
Nima Rasekh.
\newblock Cartesian fibrations and representability.
\newblock {\em arXiv preprint}, 2017.
\newblock \href{https://arxiv.org/abs/1711.03670v3}{arXiv:1711.03670v3}.

\bibitem[Ras17b]{rasekh2017left}
Nima Rasekh.
\newblock Yoneda lemma for simplicial spaces.
\newblock {\em arXiv preprint}, 2017.
\newblock \href{https://arxiv.org/abs/1711.03160v3}{arXiv:1711.03160v3}.

\bibitem[Ras21]{rasekh2021cartesiancss}
Nima Rasekh.
\newblock Cartesian fibrations of complete {S}egal spaces.
\newblock {\em arXiv preprint}, 2021.
\newblock \href{https://arxiv.org/abs/2102.05190v2}{arXiv:2102.05190v2}.

\bibitem[Ree74]{reedy1974modelstructure}
Christopher~Leonard Reedy.
\newblock {\em H{OMOLOGY} {OF} {ALGEBRAIC} {THEORIES}}.
\newblock ProQuest LLC, Ann Arbor, MI, 1974.
\newblock Thesis (Ph.D.)--University of California, San Diego.

\bibitem[Rez01]{rezk2001css}
Charles Rezk.
\newblock A model for the homotopy theory of homotopy theory.
\newblock {\em Trans. Amer. Math. Soc.}, 353(3):973--1007, 2001.

\bibitem[Rez10]{rezk2010thetanspaces}
Charles Rezk.
\newblock A {C}artesian presentation of weak {$n$}-categories.
\newblock {\em Geom. Topol.}, 14(1):521--571, 2010.

\bibitem[RV17]{riehlverity2017inftycosmos}
Emily Riehl and Dominic Verity.
\newblock Fibrations and {Y}oneda's lemma in an {$\infty$}-cosmos.
\newblock {\em J. Pure Appl. Algebra}, 221(3):499--564, 2017.

\bibitem[RV21]{riehlverity2018elements}
Emily Riehl and Dominic Verity.
\newblock Elements of $\infty$-category theory.
\newblock 2021.
\newblock \href{https://math.jhu.edu/~eriehl/elements.pdf}{Unpublished book}
  (01.02.2021).

\bibitem[sga72]{sga4}
{\em Th\'{e}orie des topos et cohomologie \'{e}tale des sch\'{e}mas. {T}ome 1:
  {T}h\'{e}orie des topos}.
\newblock Lecture Notes in Mathematics, Vol. 269. Springer-Verlag, Berlin-New
  York, 1972.
\newblock S\'{e}minaire de G\'{e}om\'{e}trie Alg\'{e}brique du Bois-Marie
  1963--1964 (SGA 4), Dirig\'{e} par M. Artin, A. Grothendieck, et J. L.
  Verdier. Avec la collaboration de N. Bourbaki, P. Deligne et B. Saint-Donat.

\bibitem[Ste17]{stevenson2017covariant}
Danny Stevenson.
\newblock Covariant model structures and simplicial localization.
\newblock {\em North-West. Eur. J. Math.}, 3:141--203, 2017.

\bibitem[To{\"e}05]{toen2005unicity}
Bertrand To{\"e}n.
\newblock Vers une axiomatisation de la th\'{e}orie des cat\'{e}gories
  sup\'{e}rieures.
\newblock {\em $K$-Theory}, 34(3):233--263, 2005.

\bibitem[Ver08]{verity2008complicial}
D.~R.~B. Verity.
\newblock Weak complicial sets. {I}. {B}asic homotopy theory.
\newblock {\em Adv. Math.}, 219(4):1081--1149, 2008.

\end{thebibliography}

\end{document}